\documentclass[11pt]{amsart}
\usepackage{amsmath,mathtools}
\usepackage{amsmath, amsthm, amssymb, amsfonts, enumerate}

\usepackage[colorlinks=true,linkcolor=blue,urlcolor=blue]{hyperref}
\usepackage{dsfont}
\usepackage{color}
\usepackage{geometry}
\usepackage{todonotes}
\usepackage{epstopdf}
\usepackage{bbm}
\usepackage{float}
\usepackage{amssymb}
\usepackage{float}
\usepackage{tikz}
\usepackage{epsfig}
\usepackage{float}
\usepackage{amssymb}
\usepackage{float}
\usepackage{tikz}
\usepackage{epsfig}

\newcommand{\stkout}[1]{\ifmmode\text{\sout{\ensuremath{#1}}}\else\sout{#1}\fi}

\newtheorem{theorem}{Theorem}[section]
\newtheorem{remark}[theorem]{Remark}
\newtheorem{assumption}[theorem]{Assumption}
\newtheorem{lemma}[theorem]{Lemma}
\newtheorem{proposition}[theorem]{Proposition}
\newtheorem{corollary}[theorem]{Corollary}
\newtheorem{df}[theorem]{Definition}
\newtheorem{problem}[theorem]{Problem}

\def \E{\mathsf{E}}

\def \P{\mathsf{P}}
\def \R{\mathbb{R}}
\def \F{\mathbb{F}}
\def \Q{\mathsf{Q}}

\def\d{\mathrm{d}}

\definecolor{red}{rgb}{1.0,0.0,0.0}

\definecolor{blu}{rgb}{0.0,0.0,1.0}

\definecolor{gre}{rgb}{0.03,0.50,0.03}

\title[Singular Control with Interconnected Dynamics]{A Singular Stochastic Control Problem with Interconnected Dynamics}
\author[Federico]{Salvatore Federico}
\author[Ferrari]{Giorgio Ferrari}
\author[Schuhmann]{Patrick Schuhmann}

\address{S.~Federico: Dipartimento di Economia Politica e Statistica, Universit\`a di Siena, Piazza san Francesco 7/8, 53100, Siena Italy}
\email{\href{mailto:salvatore.federico@unisi.it}{salvatore.federico@unisi.it}}
\address{G.~Ferrari: Center for Mathematical Economics (IMW), Bielefeld University, Universit\"atsstrasse 25, 33615, Bielefeld, Germany}
\email{\href{mailto:giorgio.ferrari@uni-bielefeld.de}{giorgio.ferrari@uni-bielefeld.de}}
\address{P.~Schuhmann: Center for Mathematical Economics (IMW), Bielefeld University, Universit\"atsstrasse 25, 33615, Bielefeld, Germany}
\email{\href{mailto:patrick.schuhmann@uni-bielefeld.de}{patrick.schuhmann@uni-bielefeld.de}}

\date{\today}

\numberwithin{equation}{section}

\begin{document}

\begin{abstract} 
In this paper we study a Markovian two-dimensional bounded-variation stochastic control problem whose state process consists of a diffusive mean-reverting component and of a purely controlled one. The main problem's characteristic lies in the interaction of the two components of the state process: the mean-reversion level of the diffusive component is an affine function of the current value of the purely controlled one. By relying on a combination of techniques from viscosity theory and free-boundary analysis, we provide the structure of the value function and we show that it satisfies a second-order smooth-fit principle. Such a regularity is then exploited in order to determine a system of functional equations solved by the two monotone continuous curves (free boundaries) that split the control problem's state space in three connected regions. Further properties of the free boundaries are also obtained.
 \end{abstract}

\maketitle

\smallskip


\smallskip

{\textbf{MSC2010 subject classification}}: 93E20, 91A55, 49L25, 49J40, 91B64.


\section{Introduction}
\label{introduction}


In this paper, we study a continuous-time stochastic control problem in which the mean-reversion level of a diffusive process $X$ is an affine function of the current level of a purely controlled one, denoted by $R$. The level of the latter can be unlimitedly increased and decreased at proportional costs. A running penalty is also faced over time, and the aim is to minimize a total expected discounted cost functional. We model such an optimization problem as a Markovian degenerate, two-dimensional \emph{singular stochastic control problem with controls of bounded variation} over an infinite time-horizon (see, e.g., \cite{CMR}, \cite{K83}, \cite{Taksar85} as early contributions on singular stochastic control problems). It is Markovian and two-dimensional since the state-variable is a two-dimensional Markov process; it is degenerate since the dynamics of the controlled process does not have any diffusive term; finally, it is a bounded-variation stochastic control problem since we interpret the cumulative amounts of increase/decrease of the level of the purely controlled process as the control variables.

The coupling between the two components of the state process makes the problem of this paper quite intricate. Our analysis is mainly devoted to the value function and the geometry of the problem's state space, being the main contribution of our work the determination of the structure of the control problem's value function $V$ and the study of its regularity. More in detail: (i) we show that the state space is split into three connected regions by two monotone curves (free boundaries); (ii) we provide the expression of the value function in each of those regions; (iii) we prove that $V$ is continuously differentiable, and admits second order mixed derivative which is continuous in the whole space (second-order smooth-fit). This latter regularity allows us to obtain a system of functional equations that are necessarily solved by the free boundaries. Further properties of the latter, such as their continuity and asymptotic limits, are also determined.
To the best of our knowledge, this is the first paper where a detailed analysis of the structure of the value function and of the geometry of the state space is provided for a two-dimensional bounded-variation stochastic control problem with interconnected dynamics.

In order to perform our analysis we do not rely on the so-called ``guess-and-verify'' approach, usually employed in the study of two-dimensional degenerate singular stochastic control problems (see, e.g., \cite{AlMotairiZ}, \cite{DeAFeMo15}, \cite{DeAFeMo19}, \cite{LokkaZervos}, \cite{LonZervos}, and \cite{MehriZervos}). In the previous works, the geometry of the state space is guessed and suitable smoothness is imposed on a candidate value function. Substantial technical effort is then required when verifying all the properties that such constructed candidate solution has to satisfy in order to provide the actual problem's solution (see, e.g., \cite{MehriZervos}). This verification step is actually even harder in our problem, given the dependency of the diffusive dynamics on the current value of the purely controlled one. For this reason we follow here a direct study of the control problem's value function and state space. First of all, by exploiting the convexity of the value function, we show that $V \in W^{2,\infty}_{\text{loc}}(\R^2;\R)$; i.e., by Sobolev's embedding, it is continuously differentiable and admits second order (weak) derivatives that are locally bounded on $\R^2$. Then - denoting by $x$ the current value of the diffusive component and by $r$ that of the controlled one - through a suitable (and not immediate) approximation procedure needed to accommodate our degenerate setting, we can employ a result of \cite{ChiaHauss00} and show that the derivative $V_r$ is the value function of a related stopping game (Dynkin game). The main characteristic of such a game is that its functional involves the derivative $V_x$ of the control problem's value function in the form of a running cost; the presence of this term is due to the coupling between the two components of the control problem's state space (see also \cite{ChiaHauss00}). The fact that $V_r$ identifies with the value of a Dynkin game, together with the convexity of $V$, allows us to obtain preliminary information about the geometry of the state space of our problem. We show that there exist two monotone boundaries that delineate the regions where $V_r$ equates (up to a sign) the marginal cost of interventions $K$ (action regions). We then move on by studying the Hamilton-Jacobi-Bellman (HJB) equation associated to $V$. This takes the form of an ordinary differential equation with the gradient constraint $-K \leq V_r \leq K$ (variational inequality), and we prove that $V$ solves it in the viscosity sense. Such a result paves the way to the determination of the structure of the value function; indeed, $V$ is shown to be a classical solution to the HJB equation in the region between the two boundaries (inaction region), and therefore it is given there in terms of the linear combination of the two strictly increasing and decreasing eigenfunctions of the infinitesimal generator of the Ornstein-Uhlenbeck process. The structure of $V$ in the two action regions is then obtained by exploiting the continuity of $V$ and the gradient constraint.

The regularity of $V$ is further improved by proving that the second-order mixed derivative, $V_{xr}$, is continuous (second-order smooth fit). This proof exploits the fact that $V$ is a viscosity solution to the HJB, as well as the preliminary properties of the free boundaries, and can be obtained along the lines of the proof of Proposition 5.3 in \cite{Federico2014} (suitably adjusted to our setting). The structure of $V$ and the second-order smooth fit property have a number of notable implications. They allow to provide the asymptotic behavior of the free boundaries and, in the relevant case of a separable running cost function, to obtain their strict monotonicity, and therefore the continuity of their inverses $g_1$ and $g_2$. These latter curves are then shown to necessarily satisfy a nonlinear system of functional equations which, in the case of decoupled dynamics, coincides with that of Proposition 5.5 in \cite{Federico2014}. However, in contrast to the lengthy analytical approach followed in \cite{Federico2014}, our way of obtaining the equations for $g_1$ and $g_2$ is fully probabilistic as it employs the local-time-space calculus of \cite{Peskir2003} and properties of one-dimensional regular diffusions (see \cite{BorodinSalminen}). Unfortunately, the highly complex structure of the equations for $g_1$ and $g_2$ makes a statement about the uniqueness of their solution far from being trivial, and we leave the study of this relevant issue for future research. 

In a final section of this paper, we show that an optimal control is given in terms of the solution (if it exists) to a suitable Skorokhod reflection problem at the boundary of the inaction region. Existence of multi-dimensional reflected diffusions is per se an interesting and not trivial question, that is linked to the regularity of the reflection boundary and direction of reflection. Under additional requirements on the running cost function $f$, we are able to find bounds on the free boundaries, and then to construct a (weak) solution to the reflection problem by following the approach of Section 5 in \cite{ChiaHauss00}. More in general, we discuss conditions on the free boundaries ensuring the existence of a two-dimensional process $(X^{\star},R^{\star})$ that is reflected at the boundary of the inaction region. In particular, global Lipschitz-regularity of the free boundaries would make the job.

The closest papers to ours are \cite{ChiaHauss00} and \cite{Federico2014}. In fact, from a mathematical point of view, our model can be seen in between that of \cite{ChiaHauss00} (see also \cite{ChiaHauss98} for a finite-horizon version) and that of \cite{Federico2014} (see also \cite{MehriZervos}). On the one hand, we propose a degenerate version of the fully two-dimensional bounded-variation stochastic control of \cite{ChiaHauss00}; on the other hand, the problem of \cite{Federico2014} can be obtained from ours when the dynamics of the two components of the state process decouple. It is exactly the degeneracy of our state process that makes the determination of the structure of the value function possible in our problem, and it is the coupling between $X$ and $R$ that makes our analysis much more involved than that in \cite{Federico2014}.
To the best of our knowledge, the only other paper dealing with a two-dimensional degenerate singular stochastic control problem where the dynamics of the two components of the state process are coupled is \cite{villeneuve}. There it is considered a dividend and investment problem for a cash constrained firm, and both a viscosity solution approach and a verification technique is employed to get qualitative properties of the value function. It is important to notice that, differently to ours, the problem in \cite{villeneuve} is not convex, thus making it hard to prove any regularity of the value function further than its continuity.

Our control problem might encompass different applications and a first one might be in the context of the central banks' optimal management of inflation. In this regards, the diffusive mean-reverting process $X$ is the level of the inflation rate, while the purely controlled process $R$ represents the key interest rate. The level of the latter can be adjusted through the central bank's monetary policy with the aim of keeping the inflation under control. Indeed, interest rates negatively affect the inflation rate: as interest rates are reduced, more people are able to borrow more money, consumers have more money to spend, and, as a consequence, economy grows and inflation raises; vice versa, if interest rates are increased, consumers are more inclined to save since the returns from savings are higher. The presence of proportional costs in our control problem might model central banks' reluctance to make large changes in the rate; on the other hand, the running cost might be, e.g., a penalization for current levels of inflation and interest rates that are misaligned with respect to fixed target levels (think of 2\% benchmark level of inflation rate planned by the European Central Bank or the U.S.\ Federal Reserve over the medium term). We refer to \cite{ChiaHauss98}, \cite{ChiaHauss00}, and \cite{JackZervos} (the latter being actually an ergodic impulse control problem) for other bounded-variation stochastic control problems motivated to the problem of inflation targeting, and to the review \cite{Woodford} and Chapter 11 of the book \cite{BlanchardF} for an economic discussion. Another problem that might be reasonably modeled in terms of the considered singular stochastic control problem comes from environmental economics. Here, $X$ represents a company's CO$_2$ emissions and $R$ is the number of production units that do not employ fossil fuel. Such a number can be adjusted by the company at proportional costs, and increasing the use of alternative fuels (i.e.\ increasing the level of $R$) negatively affects the natural equilibrium level of emissions. The aim is to minimize a total expected cost functional that also involves a running loss function penalizing any deviation of the current level of carbon emissions from a target value exogenously chosen by a regulator.

The rest of this paper is organized as follows. In Section \ref{sec:setting} we set up the problem and provide preliminary properties of the value function. The related Dynkin game is obtained in Section \ref{sec:DynkinGame}, where we also show preliminary properties of the free boundaries. Section \ref{sec:valuecharact} gives the structure of the control problem's value function, while the second-order smooth-fit property is proved in Section \ref{2ndorderSF}. Such a regularity is then used in Section \ref{sec:finalpropandeqs} for the proof of further properties of the free boundaries and the determination of the system of equations solved by the latter (cf.\ Subsection \ref{sec:eqbds}). Section \ref{sec:OC} discusses the structure of the optimal control. Finally, Appendix \ref{sec:GameCH} provides the proof of the main theorem of Section \ref{sec:DynkinGame}.


\subsection{Notation} 
\label{sec:notation}

In the rest of this paper, we adopt the following notation and functional spaces. We will use $|\,\cdot\,|$ for the Euclidean norm on any finite-dimensional space, without indicating the dimension each time for simplicity of exposition.

Given a smooth function $h:\R \to \R$, we shall write $h^{\prime}$, $h^{\prime\prime}$, etc.\ to denote its derivatives. If the function $h$ admits $k$ continuous derivatives, $k\geq1$, we shall write $h \in C^{k}(\R;\R)$, while $h\in C(\R;\R)$ if such a function is only continuous.

For a smooth function $h:\mathbb{R}^2\to \mathbb{R}$, we denote by $h_x$, $h_r$, $h_{xx}$, $h_{rr}$, etc.\ its partial derivatives. Given $k,j\in \mathbb{N}$, we let $C^{k,j}(\mathbb{R}^2;\R)$ be the class of functions $h:\mathbb{R}^2 \to \R$ which are $k$-times continuously differentiable with respect to the first variable and $j$-times continuously differentiable with respect to the second variable. If $k=j$, we shall simply write $C^{k}(\R^2;\R)$. Moreover, for an open domain $\mathcal{O} \subseteq \mathbb{R}^d$, $d\in \{1,2\}$, we shall work with the space $C^{k,\text{Lip}}_{\text{loc}}(\mathcal{O};\R)$, $k\geq1$, which consists of all the functions $h:\mathcal{O}\to \R$ that are $k$ times continuously differentiable, with locally-Lipschitz $k$th-derivative(s). 

Also, for $p \geq 1$ we shall denote by $L^{p}(\mathcal{O};\R)$ (resp.\ $L^{p}_{\text{loc}}(\mathcal{O};\R))$ the space of real-valued functions $h:\mathcal{O}\to \mathbb{R}$ such that $|h|^p$ is integrable with respect to the Lebesgue measure on $\mathcal{O}$ (resp.\ locally integrable on $\mathcal{O}$). Finally, for $k\geq1$, we shall make use of the space $W^{k,p}(\mathcal{O};\R)$ (resp.\ $W^{k,p}_{\text{loc}}(\mathcal{O};\R)$), which is the space of all the functions $h:\mathcal{O}\to \R$ that admit $k$th-order weak derivative(s) in $L^{p}(\mathcal{O};\R)$ (resp.\ $L^{p}_{\text{loc}}(\mathcal{O};\R))$).


\section{Problem Formulation and Preliminary Results}
\label{sec:setting}

\subsection{Problem formulation}
\label{sec:pb}

Let $(\Omega, \mathcal{F},\mathbb{F}:=(\mathcal{F}_t)_{t\geq0}, \P)$ be a complete filtered probability space rich enough to accommodate an $\F$-Brownian motion $W:=(W_t)_{t\geq0}$. We assume that the filtration $\F$ satisfies the usual conditions.

Introducing the (nonempty) set
\begin{align}
\label{setA}
 \mathcal{A}:=\{ &\xi:\Omega \times \mathbb{R}_+ \to \mathbb{R}:\,(\xi_t)_{t\geq0} \text{ is } \mathbb{F}\text{-adapted and such that } t \mapsto \xi_t \text{ is a.s.} \nonumber \\
 &\text{c\`{a}dl\`{a}g and (locally) of finite variation}\},
\end{align}
for any $\xi \in \mathcal{A}$ we denote by $\xi^+$ and $\xi^-$ the two nondecreasing $\mathbb{F}$-adapted c\`{a}dl\`{a}g processes providing the minimal decomposition of $\xi$; i.e.\ $\xi=\xi^+ - \xi^-$ and the (random) Borel-measures induced on $[0,\infty)$ by $\xi^+$ and $\xi^-$ have disjoint supports. In the following, for any $\xi \in \mathcal{A}$, we set $\xi^{\pm}_{0^-}=0$ a.s.\ and we denote by $|\xi|_t:=\xi^+_t + \xi^-_t$, $t\geq0$, its total variation.

Picking $\xi \in \mathcal{A}$, we then consider the purely controlled dynamics
\begin{equation}
\label{definition of r}
R_t^{r,\xi}=r + \xi^+_t - \xi^-_t, \quad t\geq 0, \qquad R_{0^-}^{r,\xi}= r \in \mathbb{R}.     
\end{equation}
Here, $\xi^+_t$ (resp.\ $\xi^-_t$) represents the cumulative increase (resp.\ decrease) of the level of $R$ made up to time $t\geq0$. Notice that we do not restrict to cumulative actions that, as functions of time, are absolutely continuous with respect to the Lebesgue measure. In fact, also lump sum and singular interventions are allowed.

The controller acts on the level of $R$ in order to adjust the long-term equilibrium level of a mean-reverting dynamics $X$. In particular, for any given $\xi \in \mathcal{A}$, the latter process evolves as 
\begin{equation}
\label{dynamics of pi und r}
\begin{cases}
\d X_t^{x,r,\xi}=\theta \Big(\mu - b R_t^{r,\xi} - X_t^{x,r,\xi}\Big) \d t+\eta \d W_t, \quad t>0, \\
X_0^{x,r,\xi}= x \in \mathbb{R},
\end{cases}
\end{equation}
where $\eta>0$ is the volatility, $\theta>0$ is the speed of mean reversion, and $\mu \in \mathbb{R}$ and $b>0$. Defining
$$\bar{\mu}(r):=\mu - b r, \quad r \in \R,$$ 
as the $R$-dependent equilibrium (or long-term mean) of $X$, the unique strong solution to \eqref{dynamics of pi und r} can be obtained by the well known method of variation of constants and is given by
\begin{equation}
\label{explicit form of  the controlled X}
X_t^{x,r,\xi}= xe^{-\theta t} + \theta e^{-\theta t} \int_0^{t} e^{\theta s}\bar{\mu}(R_s^{r,\xi})~\d s + \eta e^{-\theta t} \int_0^t e^{\theta s}~\d W_s, \quad \forall \xi \in \mathcal{A},\,\, t\geq0.
\end{equation}
The positive parameter $b$ can be seen as a measure of the impact of $R$ on $X$. Indeed, when $b=0$, the controller's actions do not affect the dynamics of $X$, which then evolves as an Ornstein-Uhlenbeck process with mean-reversion level $\mu$. 

The controller faces a running cost depending on the current values $(X_t,R_t)$. In the problem of optimal inflation management discussed in the introduction, such a cost might be thought of as a penalization for having any misalignment of inflation $X$ and key interest rate $R$ from exogenously given reference levels; for example, the monetary policy of the European Central Bank is planned for inflation rates of below, but close to, 2\% over the medium term. 

Moreover, we assume that each intervention on the process $R$ is costly, and that, in particular, controller's actions give rise to proportional costs with marginal constant cost $K>0$. Again, with reference to the central bank problem of the introduction, those costs would model the willingness of the central banks to guarantee stable interest rates, and therefore their reluctance to make large changes in the interest rate $R$. 

The controller is then faced with the problem of choosing $\xi \in \mathcal{A}$ such that, for any $(x,r) \in \mathbb{R}^2$, the cost functional 
\begin{equation}
\label{CBfunctionalt}
\mathcal{J}(x,r;\xi):=\E\bigg[\int_0^{\infty} e^{-\rho t} f(X_t^{x,r,\xi},R_t^{r,\xi}) \d t + \int_0^{\infty} e^{-\rho t} K\, \d|\xi|_t \bigg]
\end{equation}
is minimized; that is, it aims at solving
\begin{align}
\label{definition of V}
V(x,r):=\inf_{\xi \in \mathcal{A}} \mathcal{J}(x,r;\xi), \qquad (x,r) \in \mathbb{R}^2.
\end{align}

In \eqref{CBfunctionalt} and in the following, the integrals with respect to $\d |\xi|$ and $\d \xi^{\pm}$ are intended in the Lebesgue-Stieltjes' sense; in particular, for $\zeta \in \{|\xi|,\xi^+,\xi^-\}$, we set $\int_0^s (\,\cdot\,) \d \zeta_t := \int_{[0,s]}  (\,\cdot\,) \d \zeta_t$ in order to take into account a possible mass at time zero of the Borel (random) measure $\d\zeta$. Also, the parameter $\rho>0$ is a measure of the time-preferences of the controller, while the running cost function $f:\mathbb{R}^2 \to \mathbb{R}^+$ satisfies the following standing assumption.

\begin{assumption} There exists $p>1$, and $C_0,C_1,C_2>0$ such that the following hold true:
\label{ass:f} 
\begin{itemize}
\item[(i)]  $0\leq f(z) \leq C_0\big(1 + |z|\big)^p$, for every $z=(x,r)\in \R^2$; 
\item[(ii)]  for every $z=(x,r),z'=(x',r')\in\R^2$, 
$$|f(z) - f(z')| \leq C_1 \big(1 + f(z)+f(z')\big)^{1-\frac{1}{p}} |z-z'|;$$
\item[(iii)] for every $z=(x,r),z'=(x',r')\in\R^2$ and $\lambda \in(0,1)$, 
 $$0 \leq  \lambda f(z)+(1-\lambda)f(z')-f(\lambda z + (1-\lambda) z')  \leq C_2 \lambda(1-\lambda)(1+ f(z)+ f(z'))^{(1-\frac{2}{p})^+}|z-z'|^2;$$
\item[(iv)] $x \mapsto f_r(x,r)$ is nonincreasing for any $r \in \mathbb{R}$.
\end{itemize}
\end{assumption}

\begin{remark}
\label{rem:assf}
\begin{itemize}
\item[(i)] From Assumption \ref{ass:f}-(iii) it follows that $f$ is convex and locally semiconcave; hence, by Corollary 3.3.8 in \cite{CS}, it belongs to 
$$W^{2,\infty}_{loc}(\R^2;\R)=C^{1,{Lip}}_{loc}(\R^2;\R).$$
\item[(ii)] A function $f$ satisfying Assumption \ref{ass:f} is, for example, 
$$f(x,r)= \alpha(x-\tilde{x})^2 + \beta(r - \tilde{r})^2, \quad (x,r) \in \R^2,$$ 
for some $\tilde{x} \in \R$ and $\tilde{r} \in \R$, and for some constants $\alpha,\beta \geq 0$. Another choice might be to take 
$$f(x,r)=\alpha x^p \mathds{1}_{x>0} + \beta x^q \mathds{1}_{x \leq 0}, \quad (x,r) \in \R^2,$$
for some $q>p>1$ and $\alpha,\beta>0$. In the context of the optimal inflation problem, such an asymmetric function might model the higher aversion of the central bank for deflation than inflation.
\end{itemize}
\end{remark}

\begin{remark}
\label{rem:relationlit-R}
\begin{itemize}
\item[(i)] Thinking of problem \eqref{definition of V} as a (very stylized) model of optimal inflation management, one notices from \eqref{definition of r} that the key interest rates are (possibly) unbounded. This fact might be clearly debatable from a modeling point of view, but it remarkably simplifies the mathematical treatment of problem \eqref{definition of V}. Indeed, introducing exogenous bounds on the level of $R$, the dynamic programming equation (see \eqref{HJB system} below) associated to problem \eqref{definition of V} would be complemented by boundary conditions leading to a more complex analysis. However, we shall see in Proposition \ref{prop:bbounded-rate} below, that the optimal process $R$ stays bounded under certain requirements on $f$.

\item[(ii)] It is worth noticing that the restriction $b>0$ in \eqref{dynamics of pi und r} is not necessary for the subsequent analysis; in fact, all the results of this paper (up to obvious modifications) can be still deduced with the same techniques also in the case $b<0$. We have decided to consider only the case $b > 0$ just in order to simplify the exposition and accommodate the possible applications discussed in the introduction.
\end{itemize}
\end{remark}


\subsection{Preliminary Properties of the Value Function}
\label{sec:preliminary}

We now provide some preliminary properties of the value function. Their proof is classical, but those properties will play an important role in our subsequent analysis. We notice that the linear structure of the state equations yields

\begin{equation}
\label{diff}
X_t^{x,r,\xi}-X_t^{\hat{x},\hat{r},\xi}=(x-\hat{x})e^{-\theta t}+b(\hat{r}-r)(1-e^{-\theta t}), \ \ \ \forall (x,r),\,(\hat{x},\hat{r}) \in \mathbb{R}^2, \ \forall  \xi\in \mathcal{A}, \ \forall t\geq 0. 
\end{equation}

\begin{proposition}
\label{prop:Vprelim} 
Let Assumption \ref{ass:f} hold and let $p>1$ be the constant appearing in such assumption. There exist constants  $\hat{C}_0, \hat{C}_1,\hat{C}_2>0$ such that the following hold: 
\begin{itemize}
\item[(i)]  $0\leq V(z) \leq \hat{C}_0\big(1 + |z|^p\big)+ \big(-K \min\{r,0\} \wedge K\max\{r,0\}\big)$ for every $z=(x,r)\in \R^2$; 
\item[(ii)] there exists $\hat{C}_1>0$ such that, for every $z=(x,r),z'=(x',r')\in\R^2$, 
$$|V(z) - V(z')| \leq \hat{C}_1 \big(1 + |z|+|z'|\big)^{p-1} |z-z'|;$$
\item[(iii)]  for every $z=(x,r),z'=(x',r')\in\R^2$ and $\lambda \in(0,1)$, 
 $$0 \leq  \lambda V(z)+(1-\lambda)V(z')-V(\lambda z + (1-\lambda) z')  \leq \hat{C}_2 \lambda(1-\lambda) (1+|z|+|z'|)^{(p-2)^+}|z-z'|^2;$$ 
\end{itemize}
in particular, $V$ is convex and locally semiconcave, and, by Corollary 3.3.8 in \cite{CS}, it belongs to $W^{2,\infty}_{loc}(\R^2;\R)=C^{1,{Lip}}_{loc}(\R^2;\R)$.
\end{proposition}

\begin{proof}
Due to \eqref{diff}, the properties of $f$ required in (ii) and (iii) of Assumption \ref{ass:f} are straightly inherited by $V$ (see, e.g., the proof of Theorem 1 of \cite{CMR}, that can easily adapted to our infinite time-horizon setting, or that of Theorem 2.1 in \cite{ThesisChiarolla}).

We prove (i), which requires a slightly finer argument. Let $z=(x,r)\in\R^2$ and assume $r\geq 0$. Consider then the admissible control $\bar{\xi}$ such that $\bar{\xi}^+_t =0$ and $\bar{\xi}^-_t=r$ for all $t\geq0$ a.s. We then have 
$$\mathcal{J}(x,r;\bar{\xi}) = \E\left[\int_0^\infty e^{-\rho t} f\left(xe^{-\theta t} + \theta e^{-\theta t} \int_0^{t} e^{\theta s}\bar{\mu}(0)~\d s + \eta e^{-\theta t} \int_0^t e^{\theta s}~\d W_s, 0\right)\d t\right]+ K \max\{r,0\}. 
$$
Symmetrically, if $r\leq 0$, pick the admissible $\hat{\xi}$ such that $\hat{\xi}^+_t =-r$ and $\hat{\xi}^-_t=0$ for all $t\geq0$ a.s.\ and obtain
$$
\mathcal{J}(x,r;\hat{\xi}) = \E\left[\int_0^\infty e^{-\rho t} f\left(xe^{-\theta t} + \theta e^{-\theta t} \int_0^{t} e^{\theta s}\bar{\mu}(0)~\d s + \eta e^{-\theta t} \int_0^t e^{\theta s}~\d W_s, 0\right)\d t\right] - K \min\{r,0\}. 
$$
Then, since $V(x,r)\leq \mathcal{J}(x,r;\bar{\xi}) \wedge \mathcal{J}(x,r;\hat{\xi})$, the claim follows by Assumption \ref{ass:f}-(i), \eqref{diff} and standard estimates.
\end{proof}


\section{A Related Dynkin Game}
\label{sec:DynkinGame}

In this section we derive the Dynkin game (a zero-sum game of optimal stopping) associated to Problem \eqref{definition of V}. In order to simplify the notation, in the following we write $X^{x,r}$, instead of $X^{x,r,0}$, to identify the solution to \eqref{dynamics of pi und r} for $\xi \equiv 0$.

Denote by $\mathcal{T}$ the set of all $\mathbb{F}$-stopping times. For $(\sigma, \tau) \in \mathcal{T}\times\mathcal{T}$, and $(x,r) \in \mathbb{R}^2$, consider the stopping functional
\begin{align}
\label{value function of the Dynkin game}
\Psi(\sigma,\tau;x,r) & :=\E\bigg[\int_0^{\tau \wedge \sigma} e^{-\rho t}\Big(-\theta b V_x(X_t^{x,r},r) +f_r(X_t^{x,r},r) \Big)~\d t  \nonumber \\
& -e^{-\rho \tau} K \mathbbm{1}_{\{\tau < \sigma\}} + e^{-\rho \sigma} K \mathbbm{1}_{\{\tau > \sigma\}} \bigg], 
\end{align}
where $V_x$ is the partial derivative of $V$ with respect to $x$ (which exists continuous by Proposition \ref{prop:Vprelim}).

Consider now two agents (players), playing against each other and having the possibility to end the game by choosing a stopping time: Player 1 chooses a stopping time $\sigma$, while Player 2 a stopping time $\tau$. If Player 1 stops the game before Player 2, she pays $e^{-\rho \sigma} K$ to Player 2. If Player 2 stops first, then she pays $e^{-\rho \tau}K$ to Player 1. As long as the game is running, Player 1 keeps paying Player 2 at the rate $-\theta b V_x(X_t^{x,r},R_t^{r}) +f_r(X_t^{x,r},R_t^{r})$. Clearly, Player 1 aims at minimizing functional \eqref{value function of the Dynkin game}, while Player 2 at maximizing it. For any $(x,r) \in \mathbb{R}^2$, define now
\begin{equation}
\underline{u}(x,r):= \sup_{\tau \in \mathcal{T}} \inf_{\sigma \in \mathcal{T}} \Psi(\sigma, \tau; x,r), \quad
\bar{u}(x,r):= \inf_{\sigma \in \mathcal{T}} \sup_{\tau \in \mathcal{T}} \Psi(\sigma, \tau; x,r)
\end{equation}
as the lower- and upper-values of the game. Clearly, $\underline{u} \leq \overline{u}$. We say that the game has a value if $\underline{u}=\bar{u}=:u$; in such a case,
$$u(x,r)=\inf_{\sigma \in \mathcal{T}} \sup_{\tau \in \mathcal{T}} \Psi(\sigma, \tau; x,r) = \sup_{\tau \in \mathcal{T}} \inf_{\sigma \in \mathcal{T}} \Psi(\sigma, \tau; x,r).$$
Moreover, given $(x,r)\in \mathbb{R}^2$, a pair $(\sigma^{\star},\tau^{\star}):=(\sigma^{\star}(x,r),\tau^{\star}(x,r))$ is called a \emph{saddle-point} of the game if
\begin{equation}
\label{Definiton saddle point}
\Psi(\sigma^{\star},\tau;x,r) \leq \Psi(\sigma^{\star},\tau^{\star};x,r) \leq \Psi(\sigma,\tau^{\star};x,r)
\end{equation}
for all stopping times $\sigma,\tau \in \mathcal{T}$.

We then have the following theorem, whose proof follows from Theorems 3.11 and 3.13 in \cite{ChiaHauss00}, through a suitable (and not immediate) approximation procedure needed to accommodate our degenerate setting. Details are postponed to Appendix \ref{sec:GameCH}.

\begin{theorem}
\label{thm:Dynkin}
Let $(x,r) \in \mathbb{R}^2$. Then: 
\begin{itemize}
\item[(i)] the game has a value, i.e.
$$\inf_{\sigma \in \mathcal{T}} \sup_{\tau \in \mathcal{T}} \Psi(\sigma, \tau; x,r) = \sup_{\tau \in \mathcal{T}} \inf_{\sigma \in \mathcal{T}} \Psi(\sigma, \tau; x,r);$$
\item[(ii)] such a value is given by
\begin{equation}
\label{derivative of V with respect to r equals the game}
V_r(x,r)=\inf_{\sigma \in \mathcal{T}} \sup_{\tau \in \mathcal{T}} \Psi(\sigma, \tau; x,r) = \sup_{\tau \in \mathcal{T}} \inf_{\sigma \in \mathcal{T}} \Psi(\sigma, \tau; x,r).
\end{equation}
Moreover, the couple of $\mathbb{F}$-stopping times $(\tau^{\star}(x,r),\sigma^{\star}(x,r)):=(\tau^{\star},\sigma^{\star})$ such that
\begin{equation} 
\label{optimal stopping times in the game}
\sigma^{\star}:=\inf\big\{t\geq 0:\, V_r(X_t^{x,r},r) \geq K \big\}, \quad \tau^{\star}:=\inf\big\{t\geq 0:\, V_r(X_t^{x,r},r) \leq - K \big\}
\end{equation}
(with the usual convention $\inf \emptyset = + \infty$) form a saddle-point; that is,
$$\forall \tau \in \mathcal{T}  \quad \Psi(\sigma^{\star}, \tau; x,r) \leq V_r(x,r) = \Psi(\sigma^{\star}, \tau^{\star}; x,r) \leq \Psi(\sigma, \tau^{\star}; x,r) \quad \forall \sigma \in \mathcal{T}.$$ 
\end{itemize}
\end{theorem}


From \eqref{derivative of V with respect to r equals the game} it readily follows that $-K \leq V_r(x,r) \leq K$ for any $(x,r)\in\mathbb{R}^2$. Hence, defining  
\begin{equation}
\label{definiton of the three regions}
\begin{cases}
\mathcal{I}:=\left\{(x,r)\in \mathbb{R}^2:~~V_r(x,r)=-K \right\}, \\
\mathcal{C}:=\left\{(x,r)\in \mathbb{R}^2:~~-K < V_r(x,r)<K \right\}, \\
\mathcal{D}:=\left\{(x,r)\in \mathbb{R}^2:~~V_r(x,r)=K \right\},
\end{cases}
\end{equation}
we have that those regions provide a partition of $\mathbb{R}^2$. 

By continuity of $V_r$ (cf.\ Proposition \ref{prop:Vprelim}), $\mathcal{C}$ is an open set, while $\mathcal{I}$ and $\mathcal{D}$ are closed sets. Moreover,
convexity of $V$ provides the representation
$$\mathcal{C}=\{(x,r):\ b_1(x) < r < b_2(x)\},$$
$$\mathcal{I}=\{(x,r): \ r \leq b_1(x) \}, \quad \mathcal{D}=\{(x,r): \ r \geq b_2(x) \},$$
where the functions $b_1: \mathbb{R} \to \overline{\mathbb{R}}$ and $b_2: \mathbb{R} \to \overline{\mathbb{R}}$ are defined as
\begin{equation}
\label{boundary 1}
b_1(x):=\inf \{r\in \mathbb{R} \mid V_r(x,r)>-K\}=\sup \{r \in \mathbb{R} \mid V_r(x,r)=-K\}, \quad x \in \mathbb{R},
\end{equation}
\begin{equation}
\label{boundary 2}
b_2(x):=\sup \{r\in \mathbb{R} \mid V_r(x,r)<K\}=\inf \{r \in \mathbb{R} \mid V_r(x,r)=K\}, \quad x \in \mathbb{R},
\end{equation}
(with the usual conventions $\inf\emptyset = \infty$, $\inf\mathbb{R} = -\infty$, $\sup\emptyset = -\infty$, $\sup\mathbb{R} = \infty$).

\begin{lemma}
\label{monVr}
 $V_r(\cdot,r)$ is nonincreasing for all $r \in \mathbb{R}$. 
\end{lemma}
\begin{proof}
Since $x \mapsto V_x(x,r)$ is nondecreasing for any $r \in \mathbb{R}$ by convexity of $V$ (cf.\ Proposition \ref{prop:Vprelim}) and $x \mapsto f_r(x,r)$ is nonincreasing by Assumption \ref{ass:f}-(iv), we have that $\Psi(\sigma, \tau; \cdot,r)$ is nonincreasing for every $r\in\R$ and $\sigma,\tau\in\mathcal{T}$. 
Then the claim follows by \eqref{derivative of V with respect to r equals the game}.
\end{proof}

The monotonicity of $V_r$ proved above, together with its continuity, allows to obtain preliminary properties of $b_1$ and $b_2$.

\begin{proposition}
\label{Properties of the boundaries b_1 and b_2}
The following hold:
\begin{itemize}
\item[(i)] $b_1: \mathbb{R} \to \mathbb{R} \cup \{-\infty \}$, $b_2: \mathbb{R} \to \mathbb{R} \cup \{\infty \}$; 
\item[(ii)] $b_1$ and $b_2$ are nondecreasing;
\item[(iii)] $b_1(x) < b_2(x)$ for all $x \in \mathbb{R}$;
\item[(iv)] $b_1$ is right-continuous and  $b_2$ is left-continuous.
\end{itemize}
\end{proposition}

\begin{proof}
We prove each item separately.
\vspace{0.25cm}

\emph{Proof of (i)}. We argue by contradiction and we assume that there exists $x_o\in \mathbb{R}$ such that $b_1(x_o)=\infty$. Then, we have that $V_r(x_o,r)=-K$ for all $r \in \mathbb{R}$ and therefore
$$V(x_o,r+r')=V(x_o,r)-K r'$$
for all $r,r' \in \mathbb{R}$. Using now the fact that $V$ is nonnegative, and that $V(x_o,r) \leq \mathcal{J}(x_o,r;0)<\infty$ by Proposition \ref{prop:Vprelim}, one obtains
$$K r' \leq V(x_o,r) \leq \mathcal{J}(x_o,r;0)<\infty  \quad \forall r,r' \in \mathbb{R}.$$
Since the right-hand side of the latter is independent of $r'$ and bounded, we obtain a contradiction by picking $r'$ sufficiently large. A similar argument applies to show that $b_2$ takes values in $ \mathbb{R} \cup \{\infty \}$.
\vspace{0.25cm}

\emph{Proof of (ii)}. The claimed monotonicity of $b_1$ and $b_2$ easily follows by Lemma \ref{monVr}. 
\vspace{0.25cm}

\emph{Proof of (iii)}. The fact that $b_1(x) < b_2(x)$  for any $x \in \mathbb{R}$ is due to the convexity of $V$ with respect to $r$ and to the fact that $V_r(x,\cdot)$ is continuous for any $x \in \mathbb{R}$.
\vspace{0.25cm}

\emph{Proof of (iv)}. We prove the claim relative to $b_1$, as the one relative to $b_2$ can be proved analogously. Let $\varepsilon>0$. Then for $x\in \mathbb{R}$ we have $b_1(x) \leq b_1(x+\varepsilon)$, by (ii) above. Hence, also $b_1(x) \leq \lim_{\varepsilon \downarrow 0}b_1(x+\varepsilon)=:b_1(x+)$, where the last limit exists due to monotonicity of $b_1$. However, the sequence $(x+\varepsilon,b_1(x+\varepsilon))_{\varepsilon>0} \subset \mathcal{I}$, and, because $\mathcal{I}$ is closed, we therefore obtain in the limit $(x,b_1(x+)) \in \mathcal{I}$. It thus follows $b_1(x) \geq b_1(x+)$ by \eqref{boundary 1}, and the right-continuity of $b_1$ is then proved.
\end{proof}

Let us now define 
\begin{equation}
\label{barb12}
\bar{b}_1:=\sup_{ x \in \mathbb{R}}b_1(x),\quad \underline{b}_1:=\inf_{ x \in \mathbb{R}}b_1(x), \quad \bar{b}_2:=\sup_{ x \in \mathbb{R}}b_2(x),\quad 
\underline{b}_2:=\inf_{ x \in \mathbb{R}}b_2(x),
\end{equation}
together with the pseudo-inverses of $b_1$ and $b_2$ by
\begin{equation}
\label{defintion of the functions g_1 and g_2}
g_1(r):=\inf \{x \in \mathbb{R}: b_1(x)\geq r\}, \quad g_2(r):=\sup \{x \in \mathbb{R}: b_2(x)\leq r\},
\end{equation}
with the conventions $\inf\emptyset = \infty$ and $\sup\emptyset = -\infty$.

\begin{proposition}
\label{prop1:g1g2}
The following holds: 
\begin{itemize}
\item[(i)] $g_1(r)= \sup\{ x \in \mathbb{R}:V_r(x,r) > -K\}, \quad g_2(r)= \inf\{ x \in \mathbb{R}:V_r(x,r) < K\}$;
\item[(ii)] the functions $g_1,g_2$ are nondecreasing; 
\item[(iii)] $g_1(r) > g_2(r)$ for any $r \in \mathbb{R}$;
\item[(iv)] If $\bar{b}_2<\infty$, then $g_2(r)=\infty$ for all $r \geq \bar{b}_2$ and if $\underline{b}_1>-\infty$, then $g_1(r)=-\infty$ for all $r \leq \underline{b}_1$.
\end{itemize}
\end{proposition}
\begin{proof}
Claim (i) follows by definition, while (ii) is due to Proposition \ref{Properties of the boundaries b_1 and b_2}-(ii). 

Item (iii) is due to Lemma \ref{monVr} and to the continuity of $V_r(\cdot,r)$ for any $r \in \mathbb{R}$.

To show (iv), assume $\bar{b}_2<\infty$ and suppose, by contradiction, that $\lim_{r \to \infty}g_2(r)=\bar{g}<\infty$. Then, one has $b_2(x)=\infty$ for all $x \in (\bar{g},\infty)$, and this clearly contradicts $\bar{b}_2<\infty$. The statement relative to $g_1$ can be proved analogously.
\end{proof}


\section{The Structure of the Value Function}
\label{sec:valuecharact}

In the previous section we have derived a representation of the derivative $V_r$ of the value function defined in \eqref{definition of V}, and we have shown how the state space can be split in three regions, separated by nondecreasing curves. In this section, we exploit these results and we determine the structure of the value function $V$. 

For any given and fixed $r\in\mathbb{R}$, denote by $\mathcal{L}^r$ the infinitesimal generator associated to the uncontrolled process $X^{x,r,0}$. Acting on $u \in C^2(\mathbb{R};\R)$ it yields
$$\big(\mathcal{L}^r u\big)(x):=\frac{\eta^2}{2}u^{\prime \prime}(x) + \theta (\mu - br -x)u^{\prime}(x), \quad x \in \mathbb{R}.$$ 
Recall that $\bar{\mu}(r)=\mu - b r$. For frequent future use, it is worth noticing that any solution to the $r$-parametrized family of second-order ordinary differential equations (ODEs)
$$\big(\mathcal{L}^r\alpha(\cdot,r)\big)(x) - \rho\alpha(x,r) =0, \quad x \in \mathbb{R},$$
can be written as 
$$\alpha(x,r)=A(r)\psi(x-\bar{\mu}(r)) + B(r)\varphi(x-\bar{\mu}(r)), \quad x \in \mathbb{R}.$$
Here, the strictly positive functions $\psi$ and $\varphi$ are strictly increasing and decreasing fundamental solutions to the ODE
\begin{equation}
\label{ODEpsiphi}
\frac{\eta^2}{2}\zeta^{\prime \prime}(x)  - \theta x\zeta^{\prime}(x) - \rho \zeta(x) =0, \quad x \in \mathbb{R}.
\end{equation}
The functions $\psi$ and $\varphi$ are given by (see page 280 in \cite{JYC}, among others)
\begin{equation}
\label{eq:psi}
\psi(x)=e^{\frac{\theta x^2}{2\eta^2}}D_{-\frac{\rho}{\theta}}\bigg(-\frac{x}{\eta}\sqrt{2\theta}\bigg) \quad \text{and} \quad \varphi(x)=e^{\frac{\theta x^2}{2\eta^2}}D_{-\frac{\rho}{\theta}}\bigg(\frac{x}{\eta}\sqrt{2\theta}\bigg),
\end{equation}
where 
\begin{equation}
\label{eq35}
D_\beta(x):=\frac{e^{-\frac{x^2}{4}}}{\Gamma(-\beta)}\int_{0}^{\infty}t^{-\beta-1}e^{-\frac{t^2}{2}-xt}dt,\quad \beta<0,
\end{equation} 
is the Cylinder function of order $\beta$ and $\Gamma(\,\cdot\,)$ is the Euler's Gamma function (see, e.g., Chapter VIII in \cite{bateman}). Moreover, $\psi$ and $\varphi$ are strictly convex.

By the dynamic programming principle, we expect that $V$ identifies with a suitable solution to the following variational inequality
\begin{equation}
\label{HJB system}
\max\bigg\{-v_r(x,r) - K,\  v_r(x,r) - K, 
\ [(\rho-\mathcal{L}^r)v(\cdot,r)](x) - f(x,r) \bigg\}=0, \ \ \ (x,r) \in \mathbb{R}^2.
\end{equation}

By assuming that an optimal control exists, the latter can be derived by noticing that in the optimal control problem \eqref{definition of V} only three actions are possible at initial time (and, hence, at any time given the underlying Markovian framework): (i) do not intervene for a small amount of time, and then continue optimally; (ii) immediately adjust the level of $R$ via a lump sum decrease having marginal cost $K$, and then continue optimally; (iii) immediately adjust the level of $R$ via a lump sum increase having marginal cost $K$, and then continue optimally. Then, by supposing that $V$ is smooth enough, an application of It\^o's formula and a standard limiting procedure involving the mean-value theorem leads to \eqref{HJB system} (we refer to \cite{MehriZervos} for details in a related setting).

We now show that $V$ is a viscosity solution to \eqref{HJB system}. Later, this will enable us to determine the structure of $V$ (see Theorem \ref{Theorem: Structure of the Value function} below) and then to upgrade its regularity (cf.\ Theorem \ref{prop:2ndSF}) in order to derive necessary optimality conditions for the boundaries splitting the state space (cf.\ Theorem \ref{thm:eqbdsAB}).

\begin{df}
\begin{itemize}\hspace{10cm}
\item[(i)] A  function $v\in C^0(\mathbb{R}^2; \mathbb{R})$ is called a {\rm{viscosity subsolution}} to \eqref{HJB system} if, for every $(x,r) \in \mathbb{R}^2$ and every $\alpha \in C^{2,1}(\mathbb{R}^2;\mathbb{R})$ such that $v-\alpha$ attains a local maximum at $(x,r)$, it holds
$$\max\bigg\{ -\alpha_r(x,r) - K,\  \alpha_r(x,r)- K  ,\ \rho\alpha(x,r) -[\mathcal{L}^r\alpha(\cdot,r)](x) - f(x,r) \bigg\} \leq 0.$$

\item[(ii)] A function $v\in C^0(\mathbb{R}^2; \mathbb{R})$ is called a {\rm{viscosity supersolution}} to \eqref{HJB system} if, for every $(x,r) \in \mathbb{R}^2$ and every $\alpha \in C^{2,1}(\mathbb{R}^2;\mathbb{R})$ such that $v-\alpha$ attains a local minimum at $(x,r)$, it holds
$$\max\bigg\{ -\alpha_r(x,r) - K,\  \alpha_r(x,r)- K,\ \rho\alpha(x,r) - [\mathcal{L}^r\alpha(\cdot,r)](x) - f(x,r) \bigg\} \geq 0.$$

\item[(iii)]
A function $v\in C^0(\mathbb{R}^2; \mathbb{R})$ is called a {\rm{viscosity solution}} to \eqref{HJB system} if it is both a viscosity subsolution and supersolution.
\end{itemize}
\end{df}

Following the arguments developed in Theorem 5.1 in Section VIII.5 of \cite{FlemingSoner2006}, one can show the following result.

\begin{proposition}
\label{prop:Vvisc}
The value function $V$ is a viscosity solution to \eqref{HJB system}.
\end{proposition}

\begin{remark}
\label{rem:VS}
Clearly, due to Lemma 5.4 in Chapter 4 of \cite{YongZhou1999}, a viscosity solution  which lies in the class $W^{2,\infty}_{loc}(\R^2;\R)$ (as our value function does; cf.\ Proposition \ref{prop:Vprelim}-(iii)) is also a \emph{strong solution} (in the sense, e.g., of \cite{CCKS}; see the same reference also for relations between these notions of solutions); i.e., it solves \eqref{HJB system} in the pointwise sense almost everywhere. 
This observation might be used to prove -- in a more economic way, but at the price of invoking another concept of solution and the results of \cite{CCKS} --  some properties of $V$ (see also Remark \ref{rem:viscstrong} below). Nonetheless, in order to keep the paper self-contained as much as possible, we will not make use of the concept of strong solution.

Our choice of using the concept of viscosity solution is motivated by the fact that we will deal afterward (see Proposition \ref{Lemma: V is viscosity solution inside the continuation region} and Theorem \ref{Second-order smooth-fit} below) with the variational inequality \eqref{HJB system} on sets of null Lebesgue measure (regular lines). Indeed, the concept of viscosity solution still provides information on what happens on those sets, as the viscosity property holds \emph{for all} (and not merely for a.e.) points of the state space $\R^2$.
\end{remark}

For future frequent use, notice that the function
\begin{equation}
\label{eq:Vhat}
\widehat{V}(x,r) := \mathcal{J}(x,r;0)=\E\bigg[\int_0^{\infty} e^{-\rho t} f(X^{x,r}_t, r)\, \d t\bigg], \quad (x,r) \in \R^2,
\end{equation}
is finite and that, for any $r\in \R$, by Feynman-Kac's theorem it identifies with a classical particular solution to the inhomogeneous linear ODE
\begin{equation}
[(\mathcal{L}^r-\rho)q(\cdot,r)](x) + f(x,r) =0, \quad x \in \mathbb{R}.
\end{equation}
Moreover, $\widehat{V}$ is continuously differentiable with respect to $r$, given the assumed regularity of $f_x$ and $f_r$.

Recall the regions $\mathcal{C}$, $\mathcal{I}$ and $\mathcal{D}$ from \eqref{definiton of the three regions}, and that $V_r=-K$ on $\mathcal{I}$, while $V_r=K$ on $\mathcal{D}$. The next proposition provides the structure of $V$ inside $\mathcal{C}$.

\begin{proposition}
\label{Lemma: V is viscosity solution inside the continuation region} 
Recall \eqref{barb12} and let $r_o \in (\underline{b}_1,\bar{b}_2)$.
\begin{itemize}
\item[(i)] The function $V(\cdot,r_o)$ is a viscosity solution to
\begin{equation}
\label{eq: viscosity inside th continuation region}
\rho\alpha(x,r_o)-[\mathcal{L}^{r_o}\alpha(\cdot,r_o)](x) - f(x,r_o) =0, \quad x \in (g_2(r_o),g_1(r_o)).
\end{equation}

\item[(ii)] $V(\cdot,r_o) \in C^{3,Lip}_{loc}((g_2(r_o),g_1(r_o)); \mathbb{R})$.

\item[(iii)] There exist constants $A(r_o)$ and $B(r_o)$ such that for all $x \in (g_2(r_o),g_1(r_o))$
$$V(x,r_o)=A(r_o)\psi(x-\bar{\mu}(r_o)) + B(r_o) \varphi(x-\bar{\mu}(r_o)) + \widehat{V}(x,r_o),$$
where the functions $\psi$ and $\varphi$ are the fundamental strictly increasing and decreasing solutions to \eqref{ODEpsiphi} and $\widehat{V}$ is as in \eqref{eq:Vhat}.
\end{itemize}
\end{proposition}

\begin{proof}
We prove each item separately.
\vspace{0.25cm}

\emph{Proof of (i)}. We  show the subsolution property; that is, we prove that for any $x_o \in (g_2(r_o),g_1(r_o))$ and $\alpha \in C^2((g_2(r_o),g_1(r_o));\mathbb{R})$ such that $V(\cdot,r_o)-\alpha$ attains a local maximum at $x_o$ it holds that 
$$\rho\alpha(x_o,r_o) - [\mathcal{L}^{r_o}\alpha(\cdot,r_o)](x_o) - f(x_o,r_o) \leq 0.$$
First of all, we claim that
$$(V_r(x_o,r_o),\alpha^{\prime}(x_o), \alpha^{\prime \prime}(x_o)) \in D_x^{2,1,+}V(x_o,r_o),$$
where $D^{2,1,+}V(x_o,r_o)$ is the superdifferential of $V$ at $(x_o,r_o)$ of first order with respect to $r$ and of second order with respect to $x$ (see Section 5 in Chapter 4 of \cite{YongZhou1999}). This means that we have to show that
\begin{equation}
\label{eq: viscosity in conti eq0}
\limsup_{(x,r) \to (x_o,r_o)} \frac{V(x,r)-V(x_o,r_o)-V_r(x_o,r_o)(r-r_o)-\alpha^{\prime}(x_o)(x-x_o)-\frac{1}{2} \alpha^{\prime \prime}(x_o)(x-x_o)^2}{|r-r_o|+|x-x_o|^2} \leq 0.
\end{equation}

In order to prove \eqref{eq: viscosity in conti eq0}, notice first that $V(x_o,\cdot)$ is continuously differentiable, and therefore
\begin{equation}\label{caz}
\lim_{r \to r_o} \frac{V(x,r)-V(x,r_o)- V_r(x_o,r_o) (r-r_o)}{|r-r_o|} = 0 \ \ \ \ \mbox{uniformly in}  \ x\in(x_o-1,x_o+1).
 \end{equation} 

Using now Lemma 5.4 in \cite{YongZhou1999}, we have that
$$(\alpha^{\prime}(x_o), \alpha^{\prime \prime}(x_o)) \in D_x^{2,+}V(x_o,r_o),$$ where $D_x^{2,+}V(x_o,r_o)$ denotes the superdifferential of $V(\cdot,r_o)$ at $x_o$  of second order (with respect to $x$); i.e.
\begin{equation}
\label{eq: viscosity in conti eq3}
\limsup_{x \rightarrow x_o} \frac{V(x,r_o)-V(x_o,r_o)-\alpha^{\prime}(x_o)(x-x_o)-\frac{1}{2} \alpha^{\prime \prime}(x_o)(x-x_o)^2}{|x-x_o|^2} \leq 0.
\end{equation}
Adding and substracting $V(x,r_o)$ in the numerator of \eqref{eq: viscosity in conti eq0}, and using \eqref{caz} and \eqref{eq: viscosity in conti eq3}, we obtain \eqref{eq: viscosity in conti eq0}.

Using again Lemma 5.4 in \cite{YongZhou1999}, we can then construct a function $\widehat{\alpha}\in C^{2,1}(\mathbb{R}^2;\mathbb{R})$ such that $V-\widehat{\alpha}$ attains a local maximum in $(x_o,r_o)$ and
\begin{equation}
\label{eq: viscosity in conti eq5}
\left(\widehat{\alpha}_r(x_o,r_o),\widehat{\alpha}_x(x_o,r_o),\widehat{\alpha}_{xx}(x_o,r_o)\right) = 
(V_r(x_o,r_o),\alpha^{\prime}(x_o), \alpha^{\prime \prime}(x_o)).
\end{equation} 
Since $(x_o,r_o)\in \mathcal{C}$ we know that $-K < V_r(x_o,r_o) < K$, and because $V$ is a viscosity solution to \eqref{HJB system}, we obtain by \eqref{eq: viscosity in conti eq5} that
$$\rho\alpha(x_o,r_o)-[\mathcal{L}^{r_o}\alpha(\cdot,r_o)](x_o) - f(x_o,r_o) \leq 0,$$
thus completing the proof of the subsolution property. The supersolution property can be shown in an analogous way and the proof is therefore omitted. 
\vspace{0.25cm}

\emph{Proof of (ii)}. Let  $a,b\in\R$ be such that $(a,r_o),(b,r_o) \in \mathcal{C}$ and $a < b$. Introduce the Dirichlet boundary value problem
\begin{equation}
\label{auxCauchy}
\begin{cases}
(\mathcal{L}^{r_o}-\rho)q(x) + f(x,r_o) =0, \quad x \in (a,b), \\
q(a,r_o)=V(a,r_o), \quad q(b,r_o)=V(b,r_o).
\end{cases}
\end{equation}
Since $f(\cdot,r_o)\in C^{1,Lip}_{loc}((g_2(r_o),g_1(r_o));\R)$, by assumption, and $V(\cdot,r_o) \in C([a,b];\R)$, by classical results problem \eqref{auxCauchy} admits a unique classical solution $\hat{q} \in C^0([a,b];\mathbb{R}) \cap C^{3,Lip}_{loc}((a,b);\mathbb{R})$. The latter is also a viscosity solution, and by (i) above and standard uniqueness results for viscosity solutions of linear equations it must coincide with $V(\cdot,r_o)$. Hence, we have that $V(\cdot,r_o) \in C^{3,Lip}_{loc}((g_2(r_o),g_1(r_o));\mathbb{R})$ and $V(\cdot,r_o)$ is a classical solution to
$$[(\mathcal{L}^{r_o}-\rho) V(\cdot,r_o)](x) + f(x,r_o) =0, \quad x \in (g_2(r_o),g_1(r_o)),$$
given the arbitrariness of $(a,b)$ and the fact that $\mathcal{C}$ is open. 
\vspace{0.25cm}

\emph{Proof of (iii)}.  Since any solution to the homogeneous linear ODE $(\mathcal{L}^{r_o}-\rho)q=0$ is given by a linear combination of its increasing fundamental solution $\psi$ and decreasing fundamental solution $\varphi$, we conclude by (ii) and the superposition principle. 
\end{proof}

\begin{remark}\label{rem:viscstrong}
The proof of Proposition \ref{Lemma: V is viscosity solution inside the continuation region} may be considerably simplified by making use of Remark \ref{rem:VS}. Indeed, since $V$ is a $W^{2,\infty}_{loc}(\R^2;\R)$ solution to \eqref{HJB system} in the a.e.\ sense, in the open set $\mathcal{C}$ one has
$$V_{xx}(x,r)=\frac{2}{\eta^2} \big(\rho V(x,r)-\theta (\mu - br -x)V_x(x,r)-f(x,r)\big), \ \ \ \mbox{for a.e.} \ (x,r)\in\mathcal{C}.  
$$
Hence, because $V\in W^{2,\infty}_{loc}(\R^2;\R)=C^1(\R^2;\R)$, we deduce from the latter that $V_{xx}$ is continuous on $\mathcal{C}$. This implies that $V(\cdot,r_o)$ is also a classical solution to \eqref{eq: viscosity inside th continuation region} for all $r_o\in (\underline{b}_1,\bar{b}_2)$, and the other claims of the proposition follow.
\end{remark}
With the previous results at hand, we are now able to provide the structure of the value function $V$.

\begin{theorem}
\label{Theorem: Structure of the Value function}
Define the sets
\begin{equation}
\mathcal{O}_1 := \{ x \in \mathbb{R}:\, b_1(x) > -\infty \}\quad
\mathcal{O}_2 := \{ x \in \mathbb{R}:\, b_2(x) < \infty \}.
\end{equation}
There exist functions
$$A,B \in  W_{loc}^{2,\infty}((\underline{b}_1,\bar{b}_2);\mathbb{R})={C}^{1,Lip}_{loc}((\underline{b}_1,\bar{b}_2);\mathbb{R}), \quad z_{1,2}: \mathcal{O}_{1,2} \to \mathbb{R}$$
such that the value function defined in \eqref{definition of V} can be written as

\begin{equation}
\label{Structure of V}
V(x,r)=
\begin{cases}
A(r)\psi(x-\bar{\mu}(r))+B(r)\varphi(x-\bar{\mu}(r))+\widehat{V}(x,r) & \text{on } \bar{\mathcal{C}},\\
z_1(x)-K r & \text{on } \mathcal{I}, \\
z_2(x)+K r & \text{on } \mathcal{D},
\end{cases}
\end{equation}
where  $\bar{\mathcal{C}}$ denotes the closure of $\mathcal{C}$,
\begin{equation}
z_1(x):=V(x,b_1(x))+K b_1(x), \quad x \in \mathcal{O}_1
\end{equation}
and 
\begin{equation}
z_2(x):=V(x,b_2(x))-K b_2(x), \quad x \in \mathcal{O}_2.
\end{equation}
\end{theorem}

\begin{proof}
We start by deriving the structure of $V$ within $\mathcal{C}$. Using Lemma \ref{Lemma: V is viscosity solution inside the continuation region}, we already know the existence of functions $A,B:(\underline{b}_1,\bar{b}_2) \to \mathbb{R}$ such that
\begin{equation}
\label{eq1 for V}
V(x,r)=A(r)\psi(x-\bar{\mu}(r))+B(r)\varphi(x-\bar{\mu}(r))+\widehat{V}(x,r), \quad (x,r) \in \mathcal{C}.
\end{equation}
Take now $r_o \in (\underline{b}_1,\bar{b}_2)$. Since $g_1(r) > g_2(r)$ for any $r \in \mathbb{R}$ (cf.\ Proposition \ref{prop1:g1g2}-(iii)), we can find $x$ and $\tilde{x}$, $x \neq \tilde{x}$, such that $(x,r),(\tilde{x},r) \in \mathcal{C}$ for any given $r \in (r_o-\varepsilon,r_o+\varepsilon)$, for a suitably small $\varepsilon >0$. Now, by evaluating \eqref{eq1 for V} at the points $(x,r)$ and $(\tilde{x},r)$, we obtain a linear algebraic system that we can solve with respect to $A(r)$ and $B(r)$ so to obtain
\begin{equation}
\label{eq for A}
A(r)=\frac{(V(x,r)-\widehat{V}(x,r))\varphi(\tilde{x}-\bar{\mu}(r))-(V(\tilde{x},r)-\widehat{V}(\tilde{x},r)\varphi(x-\bar{\mu}(r))}{\psi(x-\bar{\mu}(r))\varphi(\tilde{x}-\bar{\mu}(r))-\psi(\tilde{x}-\bar{\mu}(r))\varphi(x-\bar{\mu}(r))},
\end{equation}
\begin{equation}\label{eq for B}
B(r)=\frac{(V(\tilde{x},r)-\widehat{V}(\tilde{x},r)\psi(x-\bar{\mu}(r))-(V(x,r)-\widehat{V}(x,r))\psi(\tilde{x}-\bar{\mu}(r))}{\psi(x-\bar{\mu}(r))\varphi(\tilde{x}-\bar{\mu}(r))-\psi(\tilde{x}-\bar{\mu}(r))\varphi(x-\bar{\mu}(r))}.
\end{equation}
Note that the denominator does not vanish due to the strict monotonicity of $\psi$ and $\varphi$, and to the fact that $x \neq \tilde{x}$. Since $r_o$ was arbitrary and $V_r$ and $\widehat{V}_r$ are continuous with respect to $r$, we therefore obtain that $A$ and $B$ belong to $W^{2,\infty}_{loc}((\underline{b}_1,\bar{b}_2);\mathbb{R})={C}_{loc}^{1,Lip}((\underline{b}_1,\bar{b}_2);\mathbb{R})$.
The structure of $V$ in the closure of $\mathcal{C}$, denoted by $\overline{\mathcal{C}}$, is then obtained by Proposition \ref{Lemma: V is viscosity solution inside the continuation region} and by recalling that $V$ is continuous on $\mathbb{R}^2$ and that $A$, $B$, and $\widehat{V}$ are also continuous.

Given the definition of $z_1$ and $z_2$, the structure of $V$ inside the regions $\mathcal{I}$ and $\mathcal{D}$ follow by \eqref{definiton of the three regions} and the continuity of $V$.
\end{proof}

\begin{remark}
\label{rem:ABW2}
Notice that, in the case when $\underline{b}_1$ (resp.\ $\bar{b}_2$) is finite, we have from \eqref{eq for A} and \eqref{eq for B} that $A$ and $B$ actually belong to $W^{2,\infty}$ up to $\underline{b}_1$ (resp.\ $\bar{b}_2$). A system of ordinary differential equations for $A$ and $B$ will be derived in \eqref{ODEAB1} and \eqref{ODEAB2} by exploiting the second-order smooth-fit property of $V$ that we prove in the next section.
\end{remark}


\section{A Second-Order Smooth-Fit Principle}
\label{2ndorderSF}

This section is devoted to the proof of a second order smooth-fit principle for the value function $V$. Precisely, we are going to show in Theorem \ref{prop:2ndSF} that the function $V_{xr}$ is jointly continuous on $\mathbb{R}^2$. The proof of such a property closely follows the arguments of Proposition 5.3 in \cite{Federico2014}; however, we provide a complete proof here in order to have a self-consistent result and also to correct a few small mistakes/typos contained in the aforementioned reference. 
 
Notice that
$$ V_{rx}(x,r)=0 \ \ \ \ ~\forall (x,r) \in   \R^2\setminus \overline{\mathcal{C}}.$$
According to that, the main result of this section establishes a smooth-fit principle for the mixed derivative. 
\begin{theorem}
\label{prop:2ndSF}
It holds
\begin{equation}
\label{Second-order smooth-fit}
\lim_{\substack{(x,r) \to \ (x_o,r_o) \\ (x,r) \in \mathcal{C}}} V_{rx}(x,r)=0 \ \ \ \ ~\forall (x_o,r_o) \in  \partial \mathcal{C}.
\end{equation}
\end{theorem}

\begin{proof}
We prove \eqref{Second-order smooth-fit} only at $\partial^1 \mathcal{C}:=\{(x,r)\in \R^2:\, V_r(x,r)=-K\}$, and we distinguish two different cases for $(x_o,r_o) \in \partial^1 \mathcal{C}$.
\vspace{0.25cm}

\emph{Case (a)}. Assume that $r_o=b_1(x_o)$. Define the function
\begin{equation}
\label{definition of V bar}
\bar{V}(x,r):= A(r)\psi(x-\bar{\mu}(r))+B(r)\varphi(x-\bar{\mu}(r))+\widehat{V}(x,r), \quad (x,r) \in \mathbb{R}^2,
\end{equation}
where $A,B$ are the functions of Theorem \ref{Theorem: Structure of the Value function}. Then, one clearly has that $\bar{V} \in C^{2,1}(\mathbb{R}^2;\mathbb{R})$. Moreover, the mixed derivative $\bar{V}_{rx}$ exists and is continuous. Since  $\bar{V}=V$ in $\bar{\mathcal{C}}$, by Lemma \ref{monVr} we conclude that
$\bar{V}_{rx}\leq 0$ { in } $\mathcal{C}$. Then by continuity of $\bar{V}_{rx}$, in order to show \eqref{Second-order smooth-fit} we have only to exclude that
\begin{equation}
\label{assuming V bar is strictly negative at the boundary}
\bar{V}_{rx}(x_o,r_o) < 0,
\end{equation}
Assume, by contradiction, \eqref{assuming V bar is strictly negative at the boundary}.
Due to the continuity of $\bar{V}$, we can then find an $\varepsilon > 0$ such that
\begin{equation}
\label{derivative of V w.r.t r and pi smaller -epsilon}
\bar{V}_{rx}(x,r) \leq -\varepsilon \ \ \ \forall (x,r)\in N_{x_o,r_o},
\end{equation}
where $N_{x_o,r_o}$ is a suitable neighborhood  of the point $(x_o,r_o) \in \partial^1 \mathcal{C}$. Notice now that $\bar{V}_r(x_o,r_o)=V_r(x_o,r_o)=-K$, because $(x_o,r_o)\in \partial^1 \mathcal{C}$, and $\bar{V}=V$ in $N_{x_o,r_o} \cap \bar{\mathcal{C}}$. Then, using \eqref{assuming V bar is strictly negative at the boundary}, we can apply the implicit function theorem to $\bar{V}_r(x,r)+K$, getting the existence of a continuous function $\bar{g}_1:(r_o-\delta,r_o+\delta) \to \mathbb{R}$, for a suitable $\delta >0$, such that $\bar{V}_r(r,\bar{g}_1(r))=-K$ in $(r_o-\delta,r_o+\delta)$. Moreover, taking into account the regularity of $A,B$, we have that $\bar{g}_1\in W^{1,\infty}(r_o-\delta,r_o+\delta)$ as
\begin{equation*}
\bar{g}_1^{\prime}(r)=-\frac{\bar{V}_{rr}(r,g_1(r))}{\bar{V}_{rx}(r,g_1(r))}  \quad \text{a.e.~ in } (r_o-\delta,r_o+\delta).
\end{equation*}
Hence, by \eqref{derivative of V w.r.t r and pi smaller -epsilon} and the fact that $A,B \in W^{2,\infty}_{\text{loc}}((\underline{b}_1,\bar{b}_2);\mathbb{R})$ (see also Remark \ref{rem:ABW2} for the case $r_o=\underline{b}_1$), there exists $M_{\varepsilon} >0$ such that
\begin{equation}\label{loclip}
|\bar{g}_1(r)-\bar{g}_1(s)| \leq M_{\varepsilon} |r-s| \ \ \ \forall r,s\in(r_o-\delta,r_o+\delta).
\end{equation}
Furthermore, recalling the definition of $g_1$ in \eqref{defintion of the functions g_1 and g_2},  $\bar{g}_1$ and $g_1$ coincide in $(r_o-\delta,r_o+\delta)$. Therefore, $g_1$ is continuous in $(r_o-\delta,r_o+\delta)$, and this fact immediately implies that $b_1$ - which is nondecreasing by Proposition \ref{Properties of the boundaries b_1 and b_2} - is actually strictly increasing in a neighborhood $(x_o-\vartheta,x_o+\vartheta)$, for a suitable $\vartheta>0$. Hence, $g_1=b_1^{-1}$ over $b_1( (x_o-\vartheta,x_o+\vartheta))$, and from \eqref{loclip} we find
\begin{equation}\label{locpill}
M_\varepsilon |b_1(x)-b_1(y)| \geq | \bar{g}_1(b_1(x))-\bar{g}_1(b_1(y))| = |x-y|,  \ \ \ \forall x,y\in(r_o-\delta,r_o+\delta).
\end{equation}
Recalling again that $b_1$ is strictly increasing in $b_1( (x_o-\vartheta,x_o+\vartheta))$, hence differentiable a.e.\ overthere, from \eqref{locpill}, we obtain
\begin{equation}
\label{area of b_1}
\exists \ b_1^{\prime} (x) \geq \frac{1}{M_{\epsilon}} \quad \forall x \in  \mathcal{Y},
\end{equation}
where $\mathcal{Y}$ is a dense set (actually of full Lebesgue measure) in $[x_0,x_0+\vartheta)$.
 
Consider now the function $[x_o,x_o+\vartheta) \ni x \mapsto V(x,r_o) \in \mathbb{R}_+$. Since $b_1$ is strictly increasing, we have that the set $K:=\{(x,r_o):x \in [x_o,x_o+\vartheta)\} \subset \mathcal{I}$, and therefore by Theorem \ref{Theorem: Structure of the Value function} that
\begin{equation}
\label{Secondorder eq 1}
V(x,r_o)=-K r_o + z_1(x) \quad \forall x \in [x_o,x_o+\vartheta).
\end{equation}
Furthermore, defining the function  
\begin{equation*}
[x_o,x_o+\vartheta) \to \mathbb{R}, \quad x \mapsto z_1(x)=V(x,b_1(x))+K b_1(x)=\bar{V}(x,b_1(x))+K b_1(x),
\end{equation*}
and applying the chain rule we get that
\begin{equation}
\label{z1prime}
\exists \ z_1^{\prime}(x)=\bar{V}_{x}(x,b_1(x))+\bar{V}_{r}(x,b_1(x))b_1^{\prime}(x)+K b_1^{\prime}(x), \ \ \ \forall x \in \mathcal{Y}.
\end{equation}
Since by definition of $b_1$ we have that $\bar{V}_{r}(x,b_1(x))=V_{r}(x,b_1(x))=-K$, we obtain from \eqref{z1prime} 
\begin{equation*}
z_1^{\prime}(x)=\bar{V}_{x}(x,b_1(x)), \quad \forall x \in \mathcal{Y}.
\end{equation*}
Using this result together with \eqref{Secondorder eq 1} we obtain existence of $V_x(x,r_o)$ for all $x \in \mathcal{Y}$ and moreover
\begin{equation}
\label{derivative of V with pi second order smooth fit}
V_{x}(x,r_o)=z_1^{\prime}(x)=\bar{V}_{x}(x,b_1(x)) \quad \forall x \in \mathcal{Y}.
\end{equation}
Using again the chain rule in \eqref{derivative of V with pi second order smooth fit} we obtain existence of $V_{xx}(x,r_o)$ for all $x \in \mathcal{Y}$ and
\begin{equation}
\label{derivative of V with pi pi second order smooth fit1}
V_{xx}(x,r_o)=z_1^{\prime \prime}(x)=\bar{V}_{x x}(x,b_1(x))+ \bar{V}_{x r}(x,b_1(x))b_1^{\prime}(x) \quad \forall x \in \mathcal{Y}.
\end{equation}
Combining \eqref{derivative of V with pi pi second order smooth fit1} with \eqref{area of b_1} and \eqref{derivative of V w.r.t r and pi smaller -epsilon} one obtains 
\begin{equation}\label{derivative of V with pi pi second order smooth fit}
V_{xx}(x,r_o) \leq \bar{V}_{x x}(x,b_1(x)) - \frac{\varepsilon}{M_{\varepsilon}} \quad \forall x \in \mathcal{Y}.
\end{equation}

Using now that $V$ is a viscosity solution to \eqref{HJB system} (in particular a subsolution) by Proposition \ref{prop:Vvisc}, that $V_{xx}$ exists for all points $x \in \mathcal{Y}$, and \eqref{derivative of V with pi second order smooth fit} and \eqref{derivative of V with pi pi second order smooth fit}, we obtain that
\begin{align}
\label{eg 2 in smooth fit}
f(x,r_o)
&\geq \rho V(x,r_o)-\theta(\mu - b r_o - x)V_{x}(x,r_o)-\frac{1}{2}\eta^2V_{x x}(x,r_o) \nonumber \\
&\geq \rho V(x,r_o)-\theta(\mu - b r_o - x)\bar{V}_{x}(x,b_1(x))-\frac{1}{2}\eta^2 \big(\bar{V}_{x x}(x,b_1(x)) - \frac{\varepsilon}{M_{\varepsilon}}\big) 
\end{align}
for all $x \in \mathcal{Y}$.
Since $\mathcal{Y}$ is dense in $[x_o,x_o+\vartheta)$, we can take a sequence $(x^n)_{n\in\mathbb{N}} \subset \mathcal{Y}$ such that $x^n \downarrow x_o$.
Evaluating \eqref{eg 2 in smooth fit} at $x=x^n$, taking limits as $n\uparrow\infty$, using the right-continuity of $b_1$, the fact that $r_o=b_1(x_o)$, and the fact that $\bar{V} \in C^{1,2}(\mathbb{R}^2;\mathbb{R})$, we obtain
\begin{equation}
\label{eq3 in second order smooth fit}
f(x_o,r_o) \geq \rho \bar{V}(x_o,r_o)-\theta(\mu - b r_o - x_o)\bar{V}_{x}(x_o,r_o)-\frac{1}{2}\eta^2 \big(\bar{V}_{x x}(x_o,r_o) - \frac{\varepsilon}{M_{\varepsilon}}\big).
\end{equation}
On the other hand, since $\rho \bar{V}(x,r)-[\mathcal{L}^r\bar{V}(\cdot,r)](x)=\rho V(x,r)-[\mathcal{L}^rV(\cdot,r)](x) = f(x,r)$ for all $(x,r) \in \mathcal{C}$, using that $\bar{V} \in C^{1,2}(\mathbb{R}^2;\mathbb{R})$ and $(x_o,r_o) \in \bar{\mathcal{C}}$, we obtain by continuity of $\bar{V}$ that
\begin{equation}
\label{eq4 in second order smooth fit}
f(x_o,r_o) = \rho \bar{V}(x_o,r_o)-\theta(\mu - b r_o - x_o)\bar{V}_{x}(x_o,r_o)-\frac{1}{2}\eta^2 \bar{V}_{x x}(x_o,r_o).
\end{equation}
Combining now \eqref{eq4 in second order smooth fit} and \eqref{eq3 in second order smooth fit} leads to $\frac{\varepsilon}{M_{\varepsilon}} \leq 0$. This gives the desired contradiction.
\vspace{0.25cm}

\emph{Case (b)}. Assume now that $x_o=g_1(r_o)$ and $r_o < b_1(x_o)$, with $b_1(x_o)<\infty$ due to Proposition \ref{Properties of the boundaries b_1 and b_2}-(i). Notice that such a case occurs if the function $b_1$ has a jump at $x_o$. Defining the segment $\Gamma:=\{(r,x_o):r \in [r_o,b_1(x_o)]\}$, it follows that $\Gamma \subset \partial^1 \mathcal{C}$. Moreover, letting again $\bar{V}$ as in \eqref{definition of V bar}, we have that $\bar{V}_r=V_r=-K$ in $\Gamma$, so that 
\begin{equation}
\label{integral for V bar}
-K-\bar{V}_r(x,r) =\bar{V}_r(x_o,r)-\bar{V}_r(x,r) =\int_{x}^{x_o} \bar{V}_{rx}(u,r)~~\d u, \ \ \forall r \in [r_o,b_1(x_o)], \ \forall x\leq x_o.
\end{equation}
Using now that $A^{\prime},B^{\prime}$ are locally Lipschitz by Theorem \ref{Theorem: Structure of the Value function}, we can take the derivative with respect to $r$ in \eqref{integral for V bar} (in the Sobolev sense) and we obtain
\begin{equation*}
-\bar{V}_{rr}(x,r)=\int_{x}^{x_o} \bar{V}_{rx r}(u,r)~\d u \quad \text{for a.e.}  \ \ r \in [r_o,b_1(x_o)], \  x\leq x_o.
\end{equation*}
The convexity of $V$ and the fact that $\bar{V}=V$ in $\bar{\mathcal{C}}$, yields $\bar{V}_{rr} \geq 0$ (again in the Sobolev sense) and therefore
\begin{equation*}
0 \geq \int_{x}^{x_o} \bar{V}_{rxr}(u,r)~\d u \quad \text{for a.e.} \quad \text{for a.e.}  \ \ r \in [r_o,b_1(x_o)], \  x\leq x_o.
\end{equation*}
Dividing now both sides by $(x_o-x)$, letting $x \to x_o$, and invoking the mean value theorem one has
\begin{equation*}
0 \geq \bar{V}_{r x r}(x_o,r) \quad \text{for a.e.}  \ \ r \in [r_o,b_1(x_o)], \  x\leq x_o.
\end{equation*}
This implies that $\bar{V}_{rx}$ is nonincreasing with respect to $r \in [r_o,b_1(x_o)]$.

If we now assume, as in Case (a) above, that $\bar{V}_{rx}(x_o,r_o) < 0$, then we must also have $\bar{V}_{r x}(x_o,b_1(x_o)) < 0$. We are therefore left with the assumption employed in the contradiction scheme of Case (a), and we can thus apply again the rationale of that case to obtain a contradiction. This completes the proof.
\end{proof}


\section{A System of Equations for the Free Boundaries}
\label{sec:finalpropandeqs}

In this section we move on by proving further properties of the free boundaries and determining a system of functional equations for them.

\subsection{Further Properties of the Free Boundaries}
\label{furtherprop}

We start by studying the limiting behavior of the free boundaries and some natural bounds.

\begin{proposition} 
\label{limit behavior of the boundaries and continuity of the g}
\begin{itemize} 
\item[(i)] Suppose that $\lim_{x \to \pm\infty}f_x(x,r)=\pm\infty$ for any $r \in \mathbb{R}$. Then 
$$\bar{b}_1=\lim_{x \uparrow \infty} b_1(x)=\infty, \quad  \underline{b}_2=\lim_{x \downarrow -\infty} b_2(x)=-\infty;$$
hence $\underline{b}_1=-\infty$ and $\bar{b}_2= \infty$.
\item[(ii)] Define 
$$\zeta_1(r):=\inf\{x \in \mathbb{R}: \theta b V_x(x,r) - f_r(x,r) - \rho K \geq 0\}, \quad r\in \mathbb{R},$$
$$\zeta_2(r):=\sup\{x \in \mathbb{R}: \theta b V_x(x,r) - f_r(x,r) + \rho K \leq 0\}, \quad r\in \mathbb{R}.$$
Then, for any $r\in \mathbb{R}$, we have  
$$g_1(r) \geq \zeta_1(r) > \zeta_2(r) \geq g_2(r).$$
\end{itemize}
\end{proposition}
\begin{proof}
We prove the two claims separately.
\vspace{0.25cm}

\emph{Proof of (i)}. Here we show that $ \lim_{x \uparrow \infty} b_1(x)=\infty$. The fact that $ \lim_{x \downarrow -\infty} b_2(x)=-\infty$ can be proved by similar arguments. We argue by contradiction assuming $\bar{b}_1:=\lim_{x \uparrow \infty} b_1(x)< \infty$. Take $r_o> \bar{b}_1$, so that $\tau^{\star}(x,r_o)=\infty$ for all $x\in\R$. Then, take $x_o>g_2(r_o)$ such that $(x_o,r_o)\in\mathcal{C}$. Clearly, every $x>x_o$ belongs to $\mathcal{C}$, and therefore, by the representation \eqref{Structure of V}, we obtain that it must be $A(r_o)=0$; indeed, otherwise, by taking limits as $x\to\infty$ and using \eqref{eq:psi}, we would contradict Proposition \ref{prop:Vprelim}. Moreover, since $\varphi'(x)\to 0$ when $x\to\infty$ (cf.\ \eqref{eq:psi}), we then have by dominated convergence 
\begin{equation}\label{stimavx}
\lim_{x\to\infty} V_x(x,r_0)=\lim_{x\to\infty} \widehat{V}_x(x,r_o)= \lim_{x\to\infty} \E\left[\int_0^\infty e^{-(\rho +\theta)t}f_x(X_t^{x,r_o}, r_o)\d t\right]=\infty.
\end{equation}
Now, setting 
$$\hat{\sigma}_x :=\inf\{t\geq 0:X_t^{x,r_o} \leq x_o\},$$ 
for $x>x_o$, we have by monotonicity of $f_r(\cdot,r)$ (cf.\ Assumption \ref{ass:f}-(iv))
 \begin{align}
 -K 
 &< V_r(x,r_o)= \inf_{\sigma \in \mathcal{T}}\E\bigg[\int_0^{ \sigma} e^{-\rho t}\Big( -b\theta V_x(X_t^{x,r_o}, r_o) +f_r(X_t^{x,r_o}, r_o)\Big)~\d t +e^{-\rho \sigma} K  \bigg] \nonumber \\
&\leq \E\bigg[\int_0^{\hat{\sigma}_x} e^{-\rho t}\Big(-b\theta V_x(X_t^{x,r_o},r_o) + f_r(x_o, r_o)\Big)~\d t + K  \bigg].
\end{align}
The latter implies 
\begin{equation}
 \label{eq in proof limits boundaries}
2K + \frac{|f_r(x_o, r_o)|}{\rho} \geq b\theta \E\bigg[\int_0^{\hat{\sigma}_x} e^{-\rho t} V_x(X_t^{x_o,r_o},r_o)~\d t\bigg].
\end{equation}
Notice that one has $\hat{\sigma}_x \to \infty$ $\P$-a.s.\ as $x\to\infty$. Hence, by dominated convergence we get a contradiction from \eqref{stimavx} and \eqref{eq in proof limits boundaries}. Finally, the fact that $\bar{b}_2 = \infty$ follows by noticing that $b_2(x) \geq b_1(x)$ for any $x \in \R$ (cf.\ Proposition \ref{Properties of the boundaries b_1 and b_2}-(iii)). 
%
\vspace{0.25cm}

\emph{Proof of (ii)}. Fix $r\in \R$. Recall that $V_r(\cdot,r) \in C(\R;\R)$ by Proposition \ref{prop:Vprelim}, $V_{rx}(\cdot,r) \in C(\R;\R)$ by Theorem \ref{prop:2ndSF}, and $V_{rxx}(\cdot,r) \in L^{\infty}_{\text{loc}}(\R;\R)$ by direct calculations on the representation of $V$ given in Theorem \ref{Theorem: Structure of the Value function}. Also, it is readily verified from \eqref{derivative of V with respect to r equals the game} that $-K \leq V_r(\cdot,r) \leq K$ on $\R^2$. Then, the semiharmonic characterization of \cite{Peskir2008} (see equations (2.27)--(2.29) therein, suitably adjusted to take care of the integral term appearing in \eqref{derivative of V with respect to r equals the game}), together with the above regularity of $V_r(\cdot,r)$, allow to obtain by standard means that $(V_r(\cdot,r),g_1(r),g_2(r))$ solves  
\begin{equation}
\label{FBP}
\begin{cases}
\big(\mathcal{L}^r - \rho\big)V_r(x,r) = \theta b V_x(x,r) - f_r(x,r) & \text{on } g_2(r) < x < g_1(r),\\
\big(\mathcal{L}^r - \rho\big)V_r(x,r) \geq \theta b V_x(x,r) - f_r(x,r) & \text{on a.e.}\ x < g_1(r), \\
\big(\mathcal{L}^r - \rho\big)V_r(x,r) \leq \theta b V_x(x,r) - f_r(x,r) & \text{on a.e.}\ x > g_2(r), \\
-K \leq V_r(x,r) \leq K & x \in \mathbb{R}, \\
V_r(g_1(r),r) = - K \quad \text{and} \quad V_r(g_2(r),r) = K,\\
V_{rx}(g_1(r),r) = 0 \quad \text{and} \quad V_{rx}(g_2(r),r) = 0.
\end{cases}
\end{equation}
In particular, we have that $V_r(x,r)=K$ for any $x<g_2(r)$, and therefore from the second equation in \eqref{FBP} we obtain
$$-\rho K \geq \theta b V_x(x,r) - f_r(x,r):=\Lambda(x,r), \quad \forall x < g_2(r).$$
Since the mapping $x \mapsto \Lambda(x,r)$ is nondecreasing for any given $r \in \mathbb{R}$ by the convexity of $V$ and the assumption on $f_r$ (cf.\ Assumption \ref{ass:f}), we obtain that
$$g_2(r) \leq \zeta_2(r)=\sup\{x \in \mathbb{R}: \theta b V_x(x,r) - f_r(x,r) + \rho K \leq 0\}.$$
An analogous reasoning also shows that 
$$g_1(r) \geq \zeta_1(r)=\inf\{x \in \mathbb{R}: \theta b V_x(x,r) - f_r(x,r) - \rho K \geq 0\}.$$
Moreover, by monotonicity and continuity of $x \mapsto \theta b V_x(x,r) - f_r(x,r)$ we have for any $r \in \R$ that
\begin{eqnarray*}
& \zeta_1(r)=\inf\{x \in \mathbb{R}: \theta b V_x(x,r) - f_r(x,r) - 2\rho K + \rho K \geq 0\} \\
& > \inf\{x \in \mathbb{R}: \theta b V_x(x,r) - f_r(x,r) + \rho K \geq 0\} \\
& = \sup\{x \in \mathbb{R}: \theta b V_x(x,r) - f_r(x,r)  + \rho K \leq 0\} = \zeta_2(r).
\end{eqnarray*}
\end{proof}

The next result readily follows from Proposition \ref{limit behavior of the boundaries and continuity of the g}-(i).
\begin{corollary}
\label{cor:finiteg12}
Suppose that $\lim_{x \to \pm\infty}f_x(x,r)=\pm\infty$ for any $r \in \mathbb{R}$. Then $g_1(r)$ and $g_2(r)$ as in \eqref{defintion of the functions g_1 and g_2} are finite for any $r \in \R$.
\end{corollary}

\begin{proposition}
\label{prop:incrbds}
Let $f$ be strictly convex with respect to $x$ for all $r\in\R$ and such that $f_{rx} = 0$. Then the boundaries $b_1$ and $b_2$ are strictly increasing.
\end{proposition}
\begin{proof}
We prove the claim only for $b_1$, since analogous arguments apply to prove it for $b_2$.
By Theorem \ref{Theorem: Structure of the Value function}, we can differentiate the first line of \eqref{Structure of V} with respect to $r$ and get by Proposition \ref{Lemma: V is viscosity solution inside the continuation region}-(i) that $V_r$ solves inside $\mathcal{C}$ the equation
\begin{equation}
\label{derivativerinsideC}
\frac{1}{2} \eta^2 V_{rxx}(x,r)+\theta(\mu - b r -x) V_{rx}(x,r) - \rho V_{r}(x,r) -\theta b V_{x}(x,r)+ \beta(r)=0,
\end{equation}
where $\beta(r):= f_r(\cdot,r)$, the latter depending only on $r$ by assumption.
By continuity, \eqref{derivativerinsideC} also holds on $\partial^1 \mathcal{C}=\{V_r=-K\}$.
Assume now, by contradiction, that the boundary $b_1$ is constant on $(x_o,x_o + \varepsilon)$, for some $x_o \in \mathbb{R}$ and some $\varepsilon>0$. 
Then $V_{rxx}=V_{rx}=0$ and $V_r=-K$ on $(x_o,x_o + \varepsilon)$. So, setting $r_o:=b_1(x_o)$, we obtain from \eqref{derivativerinsideC}
that 
$$\rho K + \beta(r_o)= \theta b V_{x}(x,r_o), \ \ \ \forall x\in(x_o,x_o+\varepsilon).$$
This means that 
$$V_{x}(\cdot,r_o)\equiv \frac{\rho K}{\theta b} + \frac{\beta(r_o)}{\theta b}, \ \ V_{xx}(\cdot,r_o)\equiv 0 \ \ \ \  \ \ \ \mbox{on} \ \  (x_o,x_o + \varepsilon).$$  On the other hand, by continuity, $V(\cdot,r_o)$ solves \eqref{eq: viscosity inside th continuation region} on $(x_o,x_o+\varepsilon)$. Therefore, 
\begin{equation}
\theta(\mu - b r_o - x) \Big[\frac{\rho K}{\theta b} + \frac{\beta(r_o)}{\theta b}\Big] - \rho V(x,r_o) + f(x,r_o) =0, \quad \forall x \in (x_o,x_o + \varepsilon).
\end{equation}
Since $f$ is strictly convex, we reach a contradiction. 
\end{proof}

Notice that the conditions on $f$ of Proposition \ref{prop:incrbds} (and of the following corollary) are satisfied, e.g., by the quadratic cost function of Remark \ref{rem:assf}. 
\begin{corollary}
\label{cor:contg12}
Let $f$ be strictly convex with respect to $x$ for all $r\in\R$ and such that $f_{rx} = 0$. Then the boundaries $g_1$ and $g_2$ defined through \eqref{defintion of the functions g_1 and g_2} are continuous.
\end{corollary}


\subsection{A System of Equations for the Free Boundaries and the Coefficients $A$ and $B$.}
\label{sec:eqbds}

Before proving the main result of this section (i.e.\ Theorem \ref{thm:eqbdsAB} below), we need to introduce some of the characteristics of the process $X^{x,r}$. Recall that $\bar{\mu}(r)=\mu - b r$, $r\in \mathbb{R}$. Then, for an arbitrary $x_o \in \mathbb{R}$, and for any given and fixed $r\in \mathbb{R}$, the scale function density of the process $X^{x,r}$ is defined as
\begin{equation}
\label{Scalefct}
S^{\prime}(x;r):=\exp\left\{-\int_{x_o}^{x} \frac{2\theta(\bar{\mu}(r)-y)}{\eta^2}~\d y\right\}, \quad x \in \mathbb{R},
\end{equation}
while the density of the speed measure is
\begin{equation}
\label{speed}
m^{\prime}(x;r):=\frac{2}{\eta^2 S^{\prime}(x;r)}, \quad x \in \mathbb{R}.
\end{equation}

For later use we also denote by $p$ the transition density of $X^{x,r}$ with respect to the speed measure; then, letting $A \mapsto \P_t(x,A;r)$, $A\in \mathcal{B}(\mathbb{R})$, $t>0$ and $r\in \mathbb{R}$, be the probability of starting at time $0$ from level $x \in \mathbb{R}$ and reaching the set $A \in \mathcal{B}(\mathbb{R})$ in $t$ units of time, we have (cf., e.g., p.\ 13 in \cite{BorodinSalminen})
$$\P_t(x,A;r)=\int_A p(t,x,y;r) m^{\prime}(y;r) \d y.$$
The density $p$ can be taken positive, jointly continuous in all variables and symmetric (i.e.\ $p(t,x,y;r)=p(t,y,x;r)$).
Furthermore, our analysis will involve the Green function $G$ that, for given and fixed $r\in \mathbb{R}$, is defined as (see again \cite{BorodinSalminen}, p.\ 19)
    \begin{equation}
		\label{Green}
		G(x,y;r):= \int_0^{\infty} e^{-\rho t} p(t,x,y;r) \d t = 
        \begin{cases}
        w^{-1} \psi(x-\bar{\mu}(r))\varphi(y-\bar{\mu}(r)) & \text{ for } x \leq y, \\
        w^{-1} \psi(y-\bar{\mu}(r))\varphi(x-\bar{\mu}(r)) & \text{ for } x \geq y,
        \end{cases}
    \end{equation}
where $w$ denotes the Wronskian between $\psi$ and $\varphi$ (normalized by $S'$).

\begin{theorem}
\label{thm:eqbdsAB}
Define $H(x,r):= - \theta b V_x(x,r) + f_r(x,r)$, $(x,r) \in \mathbb{R}^2$. The free boundaries $g_1$ and $g_2$ as in \eqref{defintion of the functions g_1 and g_2}, and the coefficients $A,B\in W^{2;\infty}_{\text{loc}}(\R;\mathbb{R})$ solve the following system of functional and ordinary differential equations
\begin{align}
    0&=\int_{g_2(r)}^{g_1(r)} \psi(y-\bar{\mu}(r))  H(y,r)  m^{\prime}(y;r)~\d y 
    + K \frac{\psi^{\prime}(g_1(r)-\bar{\mu}(r))}{S^{\prime}(g_1(r);r)} 
    + K \frac{\psi^{\prime}(g_2(r)-\bar{\mu}(r))}{S^{\prime}(g_2(r);r)}, \label{Int eq 1} \\
    0&=\int_{g_2(r)}^{g_1(r)} \varphi(y-\bar{\mu}(r))  H(y,r)  m^{\prime}(y;r)~\d y 
    + K \frac{\varphi^{\prime}(g_1(r)-\bar{\mu}(r))}{S^{\prime}(g_1(r);r)} 
    + K \frac{\varphi^{\prime}(g_2(r)-\bar{\mu}(r))}{S^{\prime}(g_2(r);r)}, \label{Int eq 2}
\end{align}
and
\begin{align}
\label{ODEAB1}
0 = & A^{\prime}(r)\psi^{\prime}(g_1(r)-\bar{\mu}(r))+bA(r)\psi^{\prime \prime}(g_1(r)-\bar{\mu}(r)) \nonumber \\
& +B^{\prime}(r)\varphi^{\prime}(g_1(r)-\bar{\mu}(r))+B(r)\varphi^{\prime \prime}(g_1(r)-\bar{\mu}(r))+\widehat{V}_{rx}(g_1(r),r),
\end{align}
\begin{align}
\label{ODEAB2}
0 = & A^{\prime}(r)\psi^{\prime}(g_2(r)-\bar{\mu}(r))+bA(r)\psi^{\prime \prime}(g_2(r)-\bar{\mu}(r)) \nonumber \\
& +B^{\prime}(r)\varphi^{\prime}(g_2(r)-\bar{\mu}(r))+B(r)\varphi^{\prime \prime}(g_2(r)-\bar{\mu}(r))+\widehat{V}_{xr}(g_2(r),r).
\end{align}
\end{theorem}

\begin{proof}
Fix $(x,r) \in \mathbb{R}^2$, and, for $n\in \mathbb{N}$, set $\tau_n:=\inf\{t \geq 0: |X^{x,r}_t| \geq n\}$, $n\in \mathbb{N}$. Propositions \ref{prop:Vprelim} and \ref{prop:2ndSF} guarantee that $V_r$ and $V_{rx}$ are continuous functions on $\mathbb{R}^2$. Moreover, direct calculations on \eqref{Structure of V} yield that $V_{rxx} \in L^{\infty}_{\text{loc}}(\mathbb{R}^2)$, upon recalling that $A,B \in W^{2,\infty}_{\text{loc}}(\R;\mathbb{R})$. Such a regularity of $V_r$ allows us to apply the local time-space calculus of \cite{Peskir2003} to the process $(e^{-\rho s}V_r(X^{x,r}_s, r))_{s\geq0}$ on the time interval $[0,\tau_n]$, take expectations (so that the term involving the stochastic integral vanishes) and obtain
\begin{align}
\label{Ito1}
    \E\Big[e^{-\rho \tau_n} V_r(X^{x,r}_{\tau_n},r)\Big]
    &=V_r(x,r)+\E\bigg[\int_0^{\tau_n}e^{-\rho s} \big[(\mathcal{L}^r-\rho)V_r(\cdot,r)\big](X^{x,r}_s) ~ \mathds{1}_{\{X^{x,r}_s \neq g_1(r)\}} \mathds{1}_{\{X^{x,r}_s \neq g_2(r)\}} ~\d s \bigg] \nonumber \\
    &=V_r(x,r)+\E\bigg[\int_0^{\tau_n} e^{-\rho s} \big(\theta b V_x(X^{x,r}_s,r) - f_r(X^{x,r}_s,r)\big) \mathds{1}_{\{ g_2(r) < X^{x,r}_s < g_1(r) \}}~\d s\bigg] 
		\nonumber \\
    &+\E\bigg[\int_0^{\tau_n}\rho K e^{-\rho s}  \mathds{1}_{\{ X^{x,r}_s > g_1(r) \}}~\d s -\int_0^{\tau_n}\rho K e^{-\rho s}  \mathds{1}_{\{ X^{x,r}_s < g_2(r) \}}~\d s \bigg].
\end{align}
Notice now that $\P(X^{x,r}_s = g_1(r))=\P(X^{x,r}_s = g_2(r))=0$, $s>0$, for any $(x,r)\in \mathbb{R}^2$ so that we can write from \eqref{Ito1} that
\begin{align}
\label{Ito2}
    V_r(x,r)
		&=\E\Big[e^{-\rho \tau_n} V_r(X^{x,r}_{\tau_n},r)\Big] - \E\bigg[\int_0^{\tau_n} e^{-\rho s} \big(\theta b V_x(X^{x,r}_s,r) - f_r(X^{x,r}_s,r)\big) \mathds{1}_{\{(X^{x,r}_s,r) \in \mathcal{C}\}}~\d s\bigg] 
		\nonumber \\
    & - \E\bigg[\int_0^{\tau_n}\rho K e^{-\rho s} \mathds{1}_{\{(X^{x,r}_s,r) \in \mathcal{I}\}}~\d s  + \int_0^{\tau_n}\rho K e^{-\rho s}  \mathds{1}_{\{(X^{x,r}_s,r) \in \mathcal{D}\}}~\d s \bigg].
\end{align}

We now aim at taking limits as $n\uparrow \infty$ in the right-hand side of the latter. To this end notice that $\tau_n \uparrow \infty$ a.s.\ when $n\uparrow \infty$, and therefore $\lim_{n\uparrow \infty} \E[e^{-\rho \tau_n} V_r(X^{x,r}_{\tau_n},r)] =0$ since $V_r \in [-K,K]$. Also, recalling \eqref{explicit form of  the controlled X}, Proposition \ref{prop:Vprelim}-(ii), and using standard estimates based on Burkholder-Davis-Gundy's inequality, one has
$$\E\bigg[\int_0^{\infty} e^{-\rho s} \big(\theta b |V_x(X^{x,r}_s,r)| + |f_r(X^{x,r}_s,r)|\big)~\d s\bigg] < + \infty.$$
Hence, thanks to the previous observations we can take limits as $n \uparrow \infty$, invoke the dominated convergence theorem, and obtain from \eqref{Ito2} that 
\begin{align}
\label{int equation for V_r}
    V_r(x,r)
    &=\E\bigg[ \int_0^{\infty}e^{-\rho s} H(X_s,r) \mathds{1}_{\{ (X^{x,r}_s,r)\in \mathcal{C} \}}~\d s \bigg]\nonumber \\
    &- \E\bigg[\int_0^{\infty}\rho K e^{-\rho s}  \mathds{1}_{\{ (X^{x,r}_s,r)\in \mathcal{I} \}}~\d s\bigg] + \E\bigg[ \int_0^{\infty} \rho K e^{-\rho s} \mathds{1}_{\{ (X^{x,r}_s,r)\in \mathcal{D} \}}~\d s \bigg] \nonumber \\
    &=: I_1(x,r)- I_2(x,r)+ I_3(x,r).
\end{align}

With the help of the Green function \eqref{Green} and Fubini's theorem, we can now rewrite each $I_i$, $i=1,2,3$, so to find
\begin{align}
\label{I1}
    I_1(x;r)
    &=\E\bigg[ \int_0^{\infty}e^{-\rho s}H(X_s,r) \mathds{1}_{\{ g_2(r) < X_s^{x,r} < g_1(r) \}}~\d s \bigg] \nonumber \\
    &= \int_0^{\infty} e^{-\rho s} \Big( \int_{-\infty}^{\infty} H(y,r) \mathds{1}_{\{ g_2(r) < y < g_1(r) \}} p(s,x,y;r) m^{\prime}(y;r)~\d y\Big) \d s  \nonumber \\
    &=\int_{-\infty}^{\infty} G(x,y;r)  H(y,r) \mathds{1}_{\{ g_2(r) < y < g_1(r) \}} m^{\prime}(y;r)~\d y \\
    &= \frac{1}{w} \varphi(x-\bar{\mu}(r))\int_{-\infty}^{x} \psi(y-\bar{\mu}(r))  H(y,r) \mathds{1}_{\{ g_2(r) < y < g_1(r) \}} m^{\prime}(y;r)~\d y \nonumber \\
    &+ \frac{1}{w} \psi(x-\bar{\mu}(r))\int_{x}^{\infty} \varphi(y-\bar{\mu}(r))  H(y,r) \mathds{1}_{\{ g_2(r) < y < g_1(r) \}} m^{\prime}(y;r)~\d y, \nonumber 
\end{align}

\begin{align}
\label{I2}
    I_2(x;r)
    &= \E\bigg[\int_0^{\infty} \rho K e^{-\rho s}  \mathds{1}_{\{(X_s,r)\in \mathcal{I} \}}~\d s\bigg] \nonumber \\
    &= \rho K \int_0^{\infty}e^{-\rho s} \Big(\int_{-\infty}^{\infty} p(s,x,y;r)  \mathds{1}_{\{ y \geq g_1(r) \}}m^{\prime}(y;r)~\d y\Big) \d s \nonumber \\
    &= \rho K \int_{-\infty}^{\infty} G(x,y;r)  \mathds{1}_{\{ y \geq g_1(r) \}}m^{\prime}(y;r) ~\d y  \\
    &= \frac{1}{w}\rho K \varphi(x-\bar{\mu}(r)) \int_{-\infty}^{x} \psi(y-\bar{\mu}(r))  \mathds{1}_{\{ y \geq g_1(r) \}}m^{\prime}(y;r) ~\d y \nonumber \\
    &+ \frac{1}{w}\rho K \psi(x-\bar{\mu}(r)) \int_{x}^{\infty} \varphi(y-\bar{\mu}(r))  \mathds{1}_{\{ y \geq g_1(r) \}}m^{\prime}(y;r) ~\d y,  \nonumber
\end{align}
and, similarly,
\begin{align}
\label{I3}
    I_3(x;r)
    &= \E\bigg[\int_0^{\infty}\rho K e^{-\rho s}  \mathds{1}_{\{ (X_s,r)\in \mathcal{D} \}}~\d s\bigg] \nonumber \\
    &= \frac{1}{w}\rho K \varphi(x-\bar{\mu}(r)) \int_{-\infty}^{x} \psi(y-\bar{\mu}(r))  \mathds{1}_{\{ y \leq g_2(r) \}}m^{\prime}(y;r) ~\d y  \\
    &+ \frac{1}{w}\rho K \psi(x-\bar{\mu}(r)) \int_{x}^{\infty} \varphi(y-\bar{\mu}(r))  \mathds{1}_{\{ y \leq g_2(r) \}}m^{\prime}(y;r) ~\d y. \nonumber
\end{align}

Now, by plugging \eqref{I1}, \eqref{I2}, and \eqref{I3} into \eqref{int equation for V_r}, and then by imposing that $V_r(g_1(r),r) = -K$ and $V_r(g_2(r),r)=K$, we obtain the two equations
$$-K=\frac{1}{w} \varphi(g_1(r)-\bar{\mu}(r))\int_{g_2(r)}^{g_1(r)} \psi(y-\bar{\mu}(r))  H(y,r)  m^{\prime}(y)~dy -I_2(g_1(r);r)+I_3(g_1(r);r)$$
and
$$K=\frac{1}{w} \psi(g_2(r)-\bar{\mu}(r))\int_{g_2(r)}^{g_1(r)} \varphi(y-\bar{\mu}(r))  H(y,r)  m^{\prime}(y)~dy -I_2(g_2(r);r)+I_3(g_2(r);r).$$
Finally, rearranging terms and using the fact that (cf.\ Chapter II in \cite{BorodinSalminen})
$$\frac{\psi^{\prime}(\cdot - \bar{\mu}(r))}{S^{\prime}(\cdot;r)}=\rho \int_{-\infty}^{\cdot} \psi(y- \bar{\mu}(r))m^{\prime}(y;r)~\d y$$
and
$$\frac{\varphi^{\prime}(\cdot - \bar{\mu}(r))}{S^{\prime}(\cdot;r)}=-\rho \int_{\cdot}^{\infty} \varphi(y- \bar{\mu}(r))m^{\prime}(y;r)~\d y,$$
yield \eqref{Int eq 1} and \eqref{Int eq 2}.

Equations \eqref{Int eq 1} and \eqref{Int eq 2} involve the coefficients $A(r)$ and $B(r)$ through the function $H$ since $V_x(x,r)=A(r)\psi^{\prime}(x-\bar{\mu}(r))+B(r)\varphi^{\prime}(x-\bar{\mu}(r))+\widehat{V}_x(x,r),$ for any $g_2(r) < x < g_1(r)$, by \eqref{Structure of V}. In order to obtain equations for $A$ and $B$, we use \eqref{Structure of V} together with the second-order smooth-fit principle $V_{rx}(g_1(r),r)=V_{rx}(g_2(r),r)=0$, and we find that, given the boundary functions $g_1$ and $g_2$, $A$ and $B$ solve the system of ODEs \eqref{ODEAB1} and \eqref{ODEAB2}.
\end{proof}

\subsubsection{Some comments on Theorem \ref{thm:eqbdsAB}}
\label{sec:commentsInteq}

Notice that equations \eqref{Int eq 1} and \eqref{Int eq 2} are consistent with those obtained in Proposition 5.5 of \cite{Federico2014}; in particular, one obtains, as a special case, those in Proposition 5.5 of \cite{Federico2014} by taking $b=0$ in ours \eqref{Int eq 1} and \eqref{Int eq 2}. However, the nature of our equations is different. While the equations in \cite{Federico2014} are algebraic, ours \eqref{Int eq 1} and \eqref{Int eq 2} are functional. Indeed, from \eqref{ODEAB1} and \eqref{ODEAB2} we see that $A$ and $B$ depend on the whole boundaries $g_1$ and $g_2$ (and not only on the points $g_1(r)$ and $g_2(r)$, for a fixed $r\in \mathbb{R}$), so that, once those coefficients are substituted into the expression for $V_{x}$, they give rise to a functional nature of \eqref{Int eq 1} and \eqref{Int eq 2}. 

In contrast to the lengthy analytic approach followed in \cite{Federico2014}, Equations \eqref{Int eq 1} and \eqref{Int eq 2} are derived via simple and handy probabilistic means using It\^o's formula and properties of linear regular diffusions. We believe that this different approach has also a methodological value. Indeed, if we would have tried to derive equations for the free boundaries imposing the continuity of $V_r$ and $V_{rx}$ at the points $(g_1(r),r)$ and $(g_2(r),r)$, $r\in \mathbb{R}$, we would have ended up with a system of complex and unhandy (algebraic and differential) equations from which it would have been difficult to observe their consistency with Proposition 5.5 of \cite{Federico2014}. In the spirit of \cite{Alvarez08} (see also \cite{Salminen85}), we also would like to mention that \eqref{Int eq 1} and \eqref{Int eq 2} can be seen as optimality conditions in terms of an integral representation based on the minimal $r$-harmonic mappings $\psi$ and $\varphi$ for the underlying diffusion $X^{x,r}$. As such, those equations could have been alternatively derived by applying the analytic representation of $r$-potentials obtained in Corollary 4.5 of \cite{LambertonZervos}.

In Theorem \ref{thm:eqbdsAB} we provide equations for the free boundaries $g_1$ and $g_2$ and for the coefficients $A$, and $B$, but we do prove uniqueness of the solution to \eqref{Int eq 1}, \eqref{Int eq 2}, \eqref{ODEAB1} and \eqref{ODEAB2}. We admit that we do not know how to establish such a uniqueness claim. Also, even if we would have uniqueness (given $g_1$ and $g_2$) of the solution to the system of ODEs \eqref{ODEAB1} and \eqref{ODEAB2}, the complexity of functional equations \eqref{Int eq 1} and \eqref{Int eq 2} is such that a proof of the uniqueness of their solution seems far to being trivial. A study of this point thus deserves a separate careful analysis that we leave for future research.


\section{On the Optimal Control}
\label{sec:OC}

Existence of an optimal control for problem \eqref{definition of V} can be shown relying on (a suitable version of) Koml\'os' theorem, by following arguments similar to those employed in the proof of Proposition 3.4 in \cite{Federico2014} (see also Theorem 3.3 in \cite{KW}). In fact, one also has uniqueness of the optimal control if the running cost function is strictly convex. 
In this section we investigate the structure of the optimal control by relating it to the solution to a Skorokhod reflection problem at $\partial \mathcal{C}$. We then discuss conditions under which such a reflection problem admits a solution.

\begin{problem}
\label{prob:Sk}
Let $(x,r) \in \overline{\mathcal{C}}$ be given and fixed. Find a process $\widehat{\xi} \in \mathcal{A}$ such that $\widehat{\xi}_{0^-}=0$ a.s.\ and, letting $(\widehat{X}^{x,r}_t, \widehat{R}^{r}_t)_{t\geq0} := (X^{x,r,\widehat{\xi}}_t, R^{r,\widehat{\xi}}_t)_{t\geq0}$ and denoting by $(\widehat{\xi}^+_t,\widehat{\xi}^{-}_t)_{t\geq0}$ its minimal decomposition, we have
\begin{equation}
\label{refl1}
(\widehat{X}^{x,r}_t, \widehat{R}^{r}_t) \in \overline{\mathcal{C}} \quad \text{for all}\,\,t\geq0,\quad \P-\text{a.s.}
\end{equation}
and
\begin{equation}
\label{refl2}
\widehat{\xi}^+_t = \int_{(0,t]} \mathds{1}_{\{\widehat{X}^{x,r}_s, \widehat{R}^{r}_s) \in \mathcal{I}\}} \d \widehat{\xi}^+_s, \qquad \widehat{\xi}^-_t = \int_{(0,t]} \mathds{1}_{\{\widehat{X}^{x,r}_s, \widehat{R}^{r}_s) \in \mathcal{D}\}} \d \widehat{\xi}^-_s.
\end{equation}
\end{problem}

The next theorem shows that a solution to Problem \ref{prob:Sk} (if it does exists) provides an optimal control.

\begin{theorem}
\label{thm:verifOC}
Let $(x,r) \in \mathbb{R}^2$ and suppose that a solution $\widehat{\xi}=\widehat{\xi}^+ -\widehat{\xi}^-$ to Problem \ref{prob:Sk} exists. Define the process $\xi^{\star}:=\xi_t^{\star,+}-\xi_t^{\star,-}$, $t\geq0$, where
\begin{equation}
\label{eq:optcontr}
\xi_t^{\star,+} := \widehat{\xi}^+_t + (x-g_1(r))^+, \qquad \xi_t^{\star,-} := \widehat{\xi}^-_t + (g_2(r)-x)^+, \quad \text{for all}\,\,t \geq0,
\end{equation}
and with $\xi^{\star}_{0^-}=0$ a.s.
Then $\xi^{\star}$ is optimal for problem \eqref{definition of V}. Moreover, if $f$ is strictly convex, it is the unique optimal control.
\end{theorem}

\begin{proof}
Being the process $\xi^{\star}$ clearly admissible, it is enough to show that 
\begin{equation}
\label{claimV}
V(x,r) \geq \E\bigg[ \int_0^{\infty} e^{-\rho t} f(X_t^{x,r,\xi^{\star}}, R_t^{r,\xi^{\star}}) \d t + \int_0^{\infty} e^{-\rho t}K \d\xi_t^{\star,+} + \int_0^{\infty} e^{-\rho t}K \d\xi_t^{\star,-} \bigg].
\end{equation}
To accomplish that, let $(K_n)_{n\in\mathbb{N}}$ be an increasing sequence of compact subsets such that $\bigcup_{n \in \mathbb{N}}K_n=\mathbb{R}^2$, and for any given $n\geq1$, define the bounded stopping time $\tau_n:=\inf\{t \geq 0: (X_t^{x,r,\xi^{\star}},R_t^{r,\xi^{\star}}) \not\in K_n \} \wedge n $. We already know by Theorem \ref{Theorem: Structure of the Value function} that $V \in C^{2,1}(\bar{\mathcal{C}};\mathbb{R})$; moreover, by construction, the process $\xi^{\star}$ is that $(X_t^{x,r,\xi^{\star}},R_t^{r,\xi^{\star}}) \in \bar{\mathcal{C}}$ for all $t\geq0$ a.s. Hence, we can apply It\^{o}'s formula on the (stochastic) time interval $[0,\tau_n$] to the process $(e^{-\rho t} V(X_{t}^{x,r,\xi^{\star}}, R_t^{r,\xi^{\star}}))_{t\geq0}$, take expectations, and obtain (upon noticing that the expectation of the resulting stochastic integral vanishes due to the continuity of $V_x$) 
\begin{align}
\label{eq1 verification}
V(x,r) & =\E\bigg[e^{-\rho \tau_n} V(X_{\tau_n}^{x,r,\xi^{\star}},R_{\tau_n}^{r,\xi^{\star}}) \bigg]
 - \E\bigg[ \int_0^{\tau_n} e^{-\rho t} [(\mathcal{L}^r-\rho)V(\cdot,R_t^{r,\xi^{\star}})](X_t^{x,r,\xi^{\star}})~\d t \bigg] \nonumber \\
& - \E\bigg[ \int_0^{\tau_n} e^{-\rho t}V_r(X_t^{x,r,\xi^{\star}},R_t^{r,\xi^{\star}})~\d\xi^{\star, \text{c}}_t \bigg] \nonumber \\
& - \E\bigg[ \sum_{0 \leq t \leq \tau_n}  e^{-\rho t} \left( V(X_t^{x,r,\xi^{\star}},R_t^{r,\xi^{\star}})-V(X_{t}^{x,r,\xi^{\star}},R_{t^-}^{r,\xi^{\star}}) \right) \bigg]. 
\end{align}
Here $\xi^{\star, \text{c}}$ denotes the continuous part of $\xi^{\star}$.
Notice now that 
$$[(\mathcal{L}^r-\rho)V(\cdot,R_t^{r,\xi^{\star}})](X_t^{x,r,\xi^{\star}})=-f(X_t^{x,r,\xi^{\star}}, R_t^{r,\xi^{\star}})$$ due to Proposition \ref{Lemma: V is viscosity solution inside the continuation region}-(i) and the fact that $V\in C^{2,1}(\bar{\mathcal{C}};\mathbb{R})$ by Theorem \ref{Theorem: Structure of the Value function}. Therefore, 
\begin{equation}
\label{eq2 verification}
\E\bigg[\int_0^{\tau_n} e^{-\rho t}[(\mathcal{L}^r-\rho)V(\cdot, R_t^{r,\xi^{\star}})](X_t^{x,r,\xi^{\star}})~\d t \bigg]= - \E\bigg[ \int_0^{\tau_n} e^{-\rho t}f(X_t^{x,r,\xi^{\star}}, R_t^{r,\xi^{\star}})~\d t \bigg].
\end{equation}

Letting $\Delta \xi^{\star,\pm}_t:=\xi^{\star,\pm}_t - \xi^{\star,\pm}_{t^-}$, $t\geq0$, notice now that 
\begin{eqnarray}
\label{eq2bis verification}
& \displaystyle V(X_t^{x,r,\xi^{\star}},R_t^{r,\xi^{\star}})-V(X_{t}^{x,r,\xi^{\star}},R_{t^-}^{r,\xi^{\star}}) = \mathds{1}_{\{\Delta \xi^{\star,+}_t>0\}}\int_0^{\Delta \xi^{\star,+}_t} V_r(X_{t}^{x,r,\xi^{\star}},R_{t^-}^{r,\xi^{\star}} + u) \d u \nonumber \\
& \displaystyle - \mathds{1}_{\{\Delta \xi^{\star,-}_t>0\}}\int_0^{\Delta \xi^{\star,-}_t} V_r(X_{t}^{x,r,\xi^{\star}},R_{t^-}^{r,\xi^{\star}} - u) \d u.
\end{eqnarray}
Since the support of the (random) measure induced on $\R_+$ by $\xi^{\star,+}$ is $\mathcal{I}$, and that of (random) the measure induced on $\R_+$ by $\xi^{\star,-}$ is $\mathcal{D}$, and $V_r=-K$ on $\mathcal{I}$ and $V_r=K$ on $\mathcal{D}$, we therefore conclude by using \eqref{eq2bis verification} that
\begin{eqnarray}
\label{eq3 verification}
& \displaystyle \E\bigg[ \int_0^{\tau_n} e^{-\rho t}V_r(X_t^{x,r,\xi^{\star}},R_t^{r,\xi^{\star}})~\d\xi^{\star, \text{c}}_t + \sum_{0 \leq t \leq \tau_n}  e^{-\rho t} \left( V(X_t^{x,r,\xi^{\star}},R_t^{r,\xi^{\star}})-V(X_{t}^{x,r,\xi^{\star}},R_{t^-}^{r,\xi^{\star}}) \right) \bigg] \nonumber \\
& = - \displaystyle \E\bigg[ \int_0^{\tau_n} e^{-\rho t}\Big(K~\d\xi^{\star,+}_t + K~\d\xi^{\star,-}_t\Big) \bigg].
\end{eqnarray}

Then using \eqref{eq2 verification} and \eqref{eq3 verification} in \eqref{eq1 verification}, we obtain
\begin{equation}
\label{eq: V bigger equal the optimal startegy}
V(x,r) \geq \E\bigg[\int_0^{\tau_n} e^{-\rho t}f(X_t^{x,r,\xi^{\star}}, R_t^{r,\xi^{\star}})~\d t + \int_0^{\tau_n} e^{-\rho t} K~d\xi^{\star,+}_t + \int_0^{\tau_n} e^{-\rho t} K~d\xi^{\star,-}_t \bigg],
\end{equation}
where the nonnegativity of $V$ has also been employed. Taking now limits as $n\uparrow \infty$ in the right-hand side of the latter, and invoking the monotone convergence theorem (due to nonnegativity of $f$ and of $K$) we obtain \eqref{claimV}.

Finally, uniqueness of the optimal control can be shown thanks to the strict convexity of $f$ by arguing as in the proof of Proposition 3.4 in the Appendix A of \cite{Federico2014}.
\end{proof}

\subsection{Construction of the Optimal Control under Additional Conditions on $f$ and Further Comments}
\label{sec:commentsOC}

The optimal control prescribes that the level of the process $R$ should be adjusted (via impulses and singularly continuous actions) in order to keep at each instant of time the joint process $(X^{x,r,\xi^{\star}}_t,R^{r,\xi^{\star}}_t)_{t\geq0}$ within the endogenously determined region $\{(x,r)\in \R^2:\, g_2(r) \leq x \leq g_1(r)\}$. Such a policy should be minimal, in the sense that only the minimal effort to accomplish such a task should be undertaken (cf.\ \eqref{refl1} and \eqref{refl2}).

\vspace{0.1cm}

A key question is now: does a solution to Problem \ref{prob:Sk} exist?
\vspace{0.1cm}

Existence of a solution to Problem \ref{prob:Sk} is per se an interesting and not trivial question. It is well known that in multi-dimensional settings the possibility of constructing a reflected diffusion at the boundary of a given domain strongly depends on the smoothness of the reflection boundary itself; sufficient conditions can be found in the early papers \cite{DupuisIshii} and  \cite{LS}. Unfortunately, our information on the boundary of the inaction region $\partial \mathcal{C}$ do not suffice to apply the results of the aforementioned works. In particular, even in the case in which $g_1$ and $g_2$ are continuous (equivalently, $b_1$ and $b_2$ are strictly increasing; see Proposition \ref{prop:incrbds} and Corollary \ref{cor:contg12}), we are not able to exclude horizontal segments of the free boundaries $g_1$ and $g_2$ (cf.\ Case (1) and Case (2) in \cite{DupuisIshii}). An alternative and more constructive way of obtaining a solution to Problem \ref{prob:Sk} is the one followed in \cite{ChiaHauss00}, where the needed reflected diffusion is constructed (weakly) by means of a Girsanov's transformation of probability measures (see Section 5 in \cite{ChiaHauss00}). The next proposition shows that this possible also in our problem when $f$ satisfies suitable additional requirements.

\begin{proposition}
\label{prop:bbounded-rate}
Suppose that there exists $C>0$ such that $|f_x|\leq C$, and that $f_r(x,r)=\beta(r)$, for some strictly increasing function $\beta:\R \to \R$ such that $\lim_{r\to\pm\infty}\beta(r)=\pm \infty$. Then there exists a weak solution (in the sense of weak solutions to SDEs) to Problem \ref{prob:Sk}.
\end{proposition}
\begin{proof}
The proof is organized in two steps.
\vspace{0.25cm}

\emph{Step 1.} We here show that $\underline{b}_2 > - \infty$ and $\bar{b}_1 < + \infty$. Using the convexity of $f(\cdot,r)$, \eqref{explicit form of  the controlled X}, and the assumed requirement on $f_x$, one easily finds that
\begin{align*}
\frac{V(x+\varepsilon,r) - V(x,r)}{\varepsilon} \leq \sup_{\xi \in \mathcal{A}}\E\bigg[\int_0^{\infty} e^{-(\rho + \theta)t} f_x(X^{x+\varepsilon,\xi,r}_t) \d t\bigg] \leq \frac{C}{\rho + \theta} =: C', \quad (x,r)\in\R^2.
\end{align*}
Analogously, for any $(x,r)\in\R^2$,
$$\frac{V(x,r) - V(x-\varepsilon,r)}{\varepsilon} \geq \inf_{\xi \in \mathcal{A}}\E\bigg[\int_0^{\infty} e^{-(\rho + \theta)t} f_x(X^{x-\varepsilon,\xi,r}_t) \d t\bigg] \geq -C'.$$
Hence, by the existence of $V_x(\cdot,r)$, we have that $|V_x| \leq C'$.

Since, by assumption, $f_r(x,r)=\beta(r)$, for some strictly increasing function $\beta:\R \to \R$ such that $\lim_{r\to\pm\infty}\beta(r)=\pm \infty$, it follows from arguments similar to those employed to prove (ii) of Proposition \ref{limit behavior of the boundaries and continuity of the g} that 
$$\{(x,r)\in \R^2:\, r \geq b_2(x)\} \subseteq \{(x,r)\in \R^2:\, r \geq \beta^{-1}(\rho K - \theta b C')\}.$$
Hence, $\underline{b}_2 > - \infty$. 

Analogously, one has that 
$$\{(x,r)\in \R^2:\, r \leq b_1(x)\} \subseteq \{(x,r)\in \R^2:\, r \leq \beta^{-1}(\theta b C' - \rho K)\};$$ 
therefore, $\bar{b}_1 < + \infty$. 
\vspace{0.25cm}

\emph{Step 2.} We here follow the approach developed in Section 5 of \cite{ChiaHauss00} in order to construct a weak solution (in the sense of weak solutions to SDEs) to Problem \ref{prob:Sk}. Let $B:=(B_t)_{t\geq0}$ be a standard Brownian motion on the filtered probability space $(\Omega,\mathcal{G},\mathbb{G}:=(\mathcal{G}_t)_{t\geq 0}, \Q)$, where $\mathbb{G}$ satisfies the usual hypotheses. The smallest such filtration is the augmented filtration generated by $B$, that we denote by $\mathbb{F}$.

Following, e.g., the arguments of Section 4.3 in \cite{Federico2014} one can construct a couple of $\mathbb{F}$-progressively measurable (since $\mathbb{F}$-adapted and right-continuous) processes $\xi^{\star}:=(\xi^{\star,+}_t,\xi^{\star,-}_t)_{t\geq0}$ such that
\begin{equation}
\label{dyn-Girs}
\begin{cases}
\d X_t=\theta \big(\mu - X_t\big) \d t+\eta \d B_t, \quad t>0,  \qquad X_0= x \in \mathbb{R}, \\
R^{\star}_t=r + \xi^{\star,+}_t - \xi^{\star,-}_t, \quad t\geq 0, \qquad \qquad \,\,  R^{\star}_{0^-}= r \in \mathbb{R}, 
\end{cases}
\end{equation}
\begin{equation}
\label{refl1-Girs}
(X_t, R^{\star}_t) \in \overline{\mathcal{C}} \quad \text{for all}\,\,t\geq0,\quad \Q-\text{a.s.},
\end{equation}
and
\begin{equation}
\label{refl2-Girs}
\xi^{\star,+}_t = \int_{(0,t]} \mathds{1}_{\{(X_s, R^{\star}_s) \in \mathcal{I}\}} \d \xi^{\star,+}_s, \qquad \xi^{\star,-}_t = \int_{(0,t]} \mathds{1}_{\{(X_s, R^{\star}_s) \in \mathcal{D}\}} \d \xi^{\star,-}_s.
\end{equation}
Since $\underline{b}_2 > - \infty$ and $\bar{b}_1 < + \infty$, there exists finite $\kappa_1,\kappa_2$ (depending on $r$) such that $\kappa_1 \leq R^{\star}_t \leq \kappa_2$ for all $t\geq 0$, $\Q$-a.s.

It thus follows by Girsanov's theorem (Corollary 5.2 in Chapter 3.5 of \cite{KS}) that the process
$$W_t:= B_t + \int_0^t \frac{b\theta}{\eta} R^{\star}_s \d s, \quad t\geq0,$$
is a standard Brownian motion on $(\Omega,\mathcal{F}^B,\mathbb{F}^B:=(\mathcal{F}^B_t)_{t\geq 0}, \P)$, where $\mathbb{F}^B$ is the (uncompleted) filtration generated by $B$, $\mathcal{F}^B:=\mathcal{F}^B_{\infty}$, and $\P$ is a probability measure on $(\Omega,\mathcal{F}^B)$ such that
$$\frac{\d \P}{\d \Q}\Big|_{\mathcal{F}^B_T}=\exp\Big(-\int_0^T \frac{b\theta}{\eta} R^{\star}_s \d B_s  - \frac{1}{2}\int_0^T \frac{b^2\theta^2}{\eta^2} \big(R^{\star}_s\big)^2 \d s\Big), \quad T < \infty.$$
Hence, $\P$-a.s., $(X_t, R^{\star}_t, \xi^{\star}_t)_{t\geq0}$ solves \eqref{definition of r} and \eqref{dynamics of pi und r}, and satisfies \eqref{refl1-Girs} and \eqref{refl2-Girs}; that is, it is a (weak) solution to Problem \ref{prob:Sk}.
\end{proof}

\begin{remark}
\label{rem:bbounded-rate}
Notice that the result of Proposition \ref{prop:bbounded-rate} is particularly relevant in the problem of optimal inflation management discussed in the introduction. Indeed, as a byproduct of Proposition \ref{prop:bbounded-rate} we have that the key interest rate stays bounded under the optimal monetary policy of the central bank.
\end{remark}

In general, the constructive approach of \cite{ChiaHauss00} also gives a strong solution to Problem \ref{prob:Sk} if one can show show that the free boundaries $b_1$ and $b_2$ are globally Lipschitz-continuous, a property that is assumed in \cite{ChiaHauss00}. In fact, in such a case, after constructing pathwise the solution to Problem \ref{prob:Sk} when $b=0$ in the dynamics of $X$ (see, e.g., Section 5 in \cite{ChiaHauss00} or Section 4.3 in \cite{Federico2014} for such a construction), one can still introduce back the linear term $-\theta b R^{\star}$ via a Girsanov's transformation. The Lipschitz property of the free boundaries does indeed guarantee that the exponential process needed for the change of measure is an exponential martingale. Hence, a weak solution to Problem \ref{prob:Sk} exists and a strong solution could then be obtained via a pathwise uniqueness claim whose proof uses, once more, the global Lipschitz-continuity of the free boundaries (see Remark 5.2 in \cite{ChiaHauss00}).

It is worth noticing that in certain obstacle problems in $\R^d$, $d\geq1$, the Lipschitz property is the preliminary regularity needed to upgrade - via a bootstrapping procedure and suitable technical conditions - the regularity of the free boundary to $C^{1,\alpha}$-regularity, for some $\alpha \in (0,1)$, and eventually to $C^{\infty}$-regularity (see \cite{CaffarelliSalsa} and \cite{Petrosyanetal}, among others, for details; see also \cite{DeAStabile} for Lipschitz-regularity results related to optimal stopping boundaries). In multi-dimensional singular stochastic control problems, Lipschitz regularity of the free boundary has been obtained, e.g., in a series of early papers by Soner and Shreve (\cite{SonerShreve1}, \cite{SonerShreve2}, and \cite{SonerShreve3}), via fine PDE techniques, and in the more recent \cite{BudhirajaRoss}, via more probabilistic arguments. In all those works the control process is monotone and the state process is a linearly controlled Brownian motion. Obtaining global Lipschitz-continuity of the free boundaries for the two-dimensional degenerate bounded-variation control problem \eqref{definition of V} is a non trivial task that we leave for future research.


\appendix

\section{Proof of Theorem \ref{thm:Dynkin}}
\label{sec:GameCH}

We want to suitably employ the results of Theorems 3.11 and 3.13 of \cite{ChiaHauss00}. However, in contrast to the fully diffusive setting of \cite{ChiaHauss00}, in our model the process $R$ is purely controlled so that the two-dimensional process $(X,R)$ is degenerate. The idea of the proof is then to perturb the dynamics of $R$ (cf.\ \eqref{definition of r}) by adding a Brownian motion $B:=(B_t)_{t\geq0}$ with volatility coefficient $\delta>0$, so to be able to apply Theorems 3.11 and 3.13 of \cite{ChiaHauss00} for any given and fixed $\delta$. The claims of Theorem \ref{thm:Dynkin} (in particular \eqref{derivative of V with respect to r equals the game}) will then follow by an opportune limit procedure as $\delta \downarrow 0$.

Let $W$ be as in Section \ref{sec:setting}, and suppose that $(\Omega, \mathcal{F}, \F, \P)$ is rich enough to accommodate also a second Brownian motion $B:=(B_t)_{t\geq0}$, independent of $W$. Then, given $(x,r)\in \mathbb{R}^2$, $\delta>0$, and $\xi \in \mathcal{A}$ (cf.\ \eqref{setA}), we denote by $(X^{\xi;\delta},R^{\xi;\delta}):=(X_t^{\xi;\delta},R_t^{\xi;\delta})_{t\geq0}$ the unique strong solution to 
		\begin{equation}
\label{eq:dynamicsdelta}
    \begin{pmatrix}
    \d R_t\\ 
		\d X_t
    \end{pmatrix}
    =\bigg[ 
    \begin{pmatrix}
    0 \\ \theta \mu + \theta b \bar{r}
    \end{pmatrix}
    +
    \begin{pmatrix}
   0 & 0\\-\theta b & -\theta
    \end{pmatrix}
    \begin{pmatrix}
    R_t\\ 
		X_t
    \end{pmatrix}
    \bigg] \d t+ 
    \begin{pmatrix}
    \delta & 0\\ 0 & \eta
    \end{pmatrix}
    \begin{pmatrix}
    \d B_t\\ 
		\d W_t
    \end{pmatrix}
    +
    \begin{pmatrix}
    1\\0
    \end{pmatrix}
    \d \xi_t.
\end{equation}
with initial data $X_{0^-}=x$ and $R_{0^-}=r$. In order to simplify the notation, in the the rest of this proof we will not stress the dependency on $(x,r)$ of the subsequent involved processes. In the case $\xi \equiv 0$, we simply write $(X^{\delta},R^{\delta}):=(X_t^{0;\delta},R_t^{0;\delta})_{t\geq0}$. 

Notice that \eqref{eq:dynamicsdelta} can be easily obtained from equation (2.2) of \cite{ChiaHauss00} by taking $c=1$, by suitably defining the matrices $b$ and $\sigma$ therein, and by setting $x_1=r$ and $x_2=x$. Then we define the perturbed optimal control problem
\begin{equation}
\label{Vdelta}
V^{\delta}(x,r):=\inf_{\xi \in \mathcal{A}}\E\bigg[ \int_0^{\infty} e^{-\rho t} f(X_t^{\xi;\delta}, R_t^{\xi;\delta}) \d t + K \int_0^{\infty} e^{-\rho t}~\d|\xi|_t \bigg].
\end{equation}

By estimates as those leading to Proposition \ref{prop:Vprelim} it can be shown that there exist constants $\tilde{C}_0, \tilde{C}_1, \tilde{C}_2$ (which are independent of $\delta$, for all $\delta$ sufficiently small) such that for any $\lambda \in (0,1)$, any $z:=(x,r) \in \mathbb{R}^2$ and $z':=(x',r')\in \mathbb{R}^2$, we have  
\begin{itemize}
\item[(i)] $0 \leq V^{\delta}(z) \leq \tilde{C}_0\big(1 + |z|\big)^p$,
\item[(ii)] $|V^{\delta}(z) - V^{\delta}(z')| \leq \tilde{C}_1 \big(1 + |z|+|z'|\big)^{p-1} |z-z'|$,
\item[(iii)] $0 \leq \lambda V^{\delta}(z) + (1-\lambda)V^{\delta}(z') - V^{\delta}(\lambda z + (1-\lambda)z') \leq \tilde{C}_2 \lambda(1-\lambda) \big(1 + |z| + |z'|\big)^{(p-2)^+}|z-z'|^2$,
\end{itemize}
where $p>1$ is the same of Assumption \ref{ass:f}. 
Hence $V^{\delta}$ is convex and locally semiconcave, and therefore $V^{\delta}\in W^{2,\infty}_{\text{loc}}(\mathbb{R}^2;\mathbb{R})$. In particular, there exists a version of $V^{\delta}\in C^{1,\text{Lip}}_{\text{loc}}(\mathbb{R}^2;\mathbb{R})$.

Let $(X^{\xi}_t, R^{\xi}_t)_{t\geq0}:=(X^{\xi;0}_t,R^{\xi;0}_t)_{t\geq0}$. By \eqref{definition of r}, \eqref{explicit form of  the controlled X}, and \eqref{eq:dynamicsdelta} one easily finds for $p \in [1,\infty)$
$$\E[|(X^{\xi;\delta}_t, R^{\xi;\delta}_t) - (X^{\xi}_t,R^{\xi}_t)|^p] \leq C_t\delta^{p}, \quad \forall \xi \in \mathcal{A}\,\, \text{and}\,\, t\geq 0,$$
for some $C_t$ that is at most of polynomial growth with respect to $t$. Using now the latter and Assumption \ref{ass:f}-(ii), it can be shown that $V^{\delta}(x,r) \rightarrow V(x,r)$ as $\delta \downarrow 0$ for each $(x,r) \in \mathbb{R}^2$. Let $\mathcal{B}_N:=\{z \in \mathbb{R}^2:\, |z| < N\}$, for some $N>0$. Since items (i)-(iii) above imply that $V^{\delta} \in W^{2,p}(\mathcal{B}_N)$ for any $p>2$ and $W^{2,p}(\mathcal{B}_N)$ is reflexive, there exists a sequence $\delta_n \downarrow 0$ as $n \uparrow \infty$ such that $V^{\delta_n}$ converges weakly in $W^{2,p}(\mathcal{B}_N)$. Because $V^{\delta_n} \rightarrow V$ pointwise and weak limits are unique, we have that $V^{\delta_n} \rightharpoonup V$ weakly in $W^{2,p}(\mathcal{B}_N)$. Since the embedding $W^{2,p}(\mathcal{B}_N) \hookrightarrow C^1(\mathcal{B}_N)$ is compact for $p>2$ (2 being the dimension of our space), it follows that
\begin{equation}
\label{sebseq1}
    V^{\delta_n} \to V \text{ locally uniformly in } \mathbb{R}^2,
\end{equation}
\begin{equation}
\label{sebseq2}
    V_x^{\delta_n} \to V_x \text{ locally uniformly in } \mathbb{R}^2, 
\end{equation}
and
\begin{equation}
\label{sebseq3}
    V_r^{\delta_n} \to V_r \text{ locally uniformly in } \mathbb{R}^2.
\end{equation}

Moreover, by Theorem 3.11 in \cite{ChiaHauss00} (easily adjusted to take care of our general convex function $f$ satisfying Assumption \ref{ass:f}, and upon noticing that $b_{11}=0$ in our setting, cf.\ \eqref{eq:dynamicsdelta}) we have that $V^{\delta}_r$ is the unique (given $V^{\delta}_x$) solution to the pointwise variational inequality:
\begin{equation}
\label{VI-Vdelta}
    \begin{cases}
    V_r^{\delta} \in W_{\text{loc}}^{2,q}(\mathbb{R}^2),~\forall q\geq 2,~ \quad -K \leq V_r^{\delta} \leq K \quad \text{a.e.\ in } \mathbb{R}^2, \\
    (\mathcal{L}^r-\rho)V_r^{\delta} \leq  \theta b V_x^{\delta} - f_r(x,r) \quad \text{a.e.\ in } \mathcal{I}^{\delta}, \\
    (\mathcal{L}^r-\rho)V_r^{\delta} \geq  \theta b V_x^{\delta} - f_r(x,r)\quad \text{a.e.\ in } \mathcal{D}^{\delta}, \\
    (\mathcal{L}^r-\rho)V_r^{\delta} =  \theta b V_x^{\delta} - f_r(x,r)   \quad \text{a.e.\ in } \mathcal{C}^{\delta}, \\
    \end{cases}
\end{equation}
where we have set
\begin{equation*}
\mathcal{I}^{\delta}:=\left\{(x,r)\in \mathbb{R}^2:~~V^{\delta}_r(x,r)=-K \right\}, \quad \mathcal{D}^{\delta}:=\left\{(x,r)\in \mathbb{R}^2:~~V^{\delta}_r(x,r)=K \right\},
\end{equation*}
and
\begin{equation*}
\mathcal{C}^{\delta}:=\left\{(x,r)\in \mathbb{R}^2:~~-K < V^{\delta}_r(x,r)<K \right\}.
\end{equation*}

Define
\begin{equation}
    \tau^{\star;\delta}:= \inf\{t\geq 0: V_r^{\delta}(X_t^{\delta}, R_t^{\delta})\leq -K\}, 
\end{equation}
\begin{equation}
    \sigma^{\star,\delta}:= \inf\{t\geq 0: V_r^{\delta}(X_t^{\delta}, R_t^{\delta})\geq K\},
\end{equation}
\begin{equation}
    \tau^{\star}:= \inf\{t\geq 0: V_r(X_t, r)\leq -K\}, 
\end{equation}
\begin{equation}
    \sigma^{\star}:= \inf\{t\geq 0: V_r(X_t, r)\geq K\}, 
\end{equation}
as well as, for a given $M>0$,
\begin{equation}
 \tau_M^{\delta}:= \inf\{t\geq 0: |X^{\delta}_t| + |R^{\delta}_t| \geq M\}, 
\end{equation}
\begin{equation}
 \tau_M:= \inf\{t\geq 0: |X_t| + |r| \geq M\}.
\end{equation}

Now, by \eqref{VI-Vdelta} we know that for each $\delta>0$ given and fixed, $V_r^{\delta}$ is regular enough to apply a weak version of It\^{o}'s lemma (see, e.g., Theorem 8.5 at p.\ 185 of \cite{BL}) so that for any stopping time $\zeta$ and some fixed $T>0$ one obtains
\begin{align}
\label{Vrdelta}
V_r^{\delta}(x,r) = & \E\bigg[- \int_0^{\tau_M^{\delta} \wedge \tau_M \wedge \zeta \wedge T} e^{-\rho s}(\mathcal{L}^r-\rho)V_r^{\delta}(X_s^{\delta},R_s^{\delta})\, \d s \nonumber \\
& \hspace{0.25cm} + e^{-\rho (\tau_M^{\delta}\wedge \tau_M \wedge \zeta \wedge T)} V_r^{\delta}\Big(X_{\tau_M^{\delta} \wedge \tau_M \wedge \zeta \wedge T}^{\delta}, R_{\tau_M^{\delta}\wedge \tau_M \wedge \zeta \wedge T}^{\delta}\Big)\bigg].
\end{align}

Given an $\F$-stopping time $\tau$, set $\zeta:=\sigma^{\star,\delta}\wedge \sigma^{\star}\wedge \tau$ in \eqref{Vrdelta}, and use that $V^{\delta}$ solves a.e.\ the variational inequality \eqref{VI-Vdelta} to find
\begin{align}
\label{limitunif1}
    V_r^{\delta}(x,r)
    &\geq \E\bigg[ \int_0^{\tau_M^{\delta}\wedge \tau_M \wedge \sigma^{\star,\delta}\wedge \sigma^{\star}\wedge \tau \wedge T} e^{-\rho s}\big( -\theta b V_x^{\delta}(X_s^{\delta}, R_s^{\delta}) + f_r(X_s^{\delta}, R_s^{\delta})\big)\, \d s \nonumber \\ 
		& \hspace{0.25cm} + e^{-\rho (\tau_M^{\delta}\wedge \tau_M \wedge \sigma^{\star,\delta}\wedge \sigma^{\star}\wedge \tau \wedge T)} V_r^{\delta}\big(X_{\tau_M^{\delta}\wedge \tau_M \wedge \sigma^{\star,\delta}\wedge \sigma^{\star}\wedge \tau \wedge T}^{\delta}, R_{\tau_M^{\delta}\wedge \tau_M \wedge \sigma^{\star,\delta}\wedge \sigma^{\star}\wedge \tau \wedge T}^{\delta}\big)\bigg] \\
    & \geq \E\bigg[ \int_0^{\tau_M^{\delta}\wedge \tau_M \wedge \sigma^{\star,\delta}\wedge \sigma^{\star}\wedge \tau \wedge T} e^{-\rho s}(- \theta b V_x^{\delta}(X_s^{\delta}, R_s^{\delta}) + f_r(X_s^{\delta}, R_s^{\delta})\big)\, \d s \nonumber \\ 
    & + \mathds{1}_{\{ \sigma^{\star,\delta} < \tau_M^{\delta}\wedge \tau_M \wedge \sigma^{\star}\wedge \tau \wedge T \}} e^{-\rho \sigma^{\star,\delta}}K - \mathds{1}_{\{ \tau < \tau_M^{\delta}\wedge \tau_M \wedge \sigma^{\star,\delta}\wedge \sigma^{\star}\wedge T\}} e^{-\rho \tau} K \nonumber \\
    & + \mathds{1}_{\{ \tau_M^{\delta}\wedge \tau_M \wedge \sigma^{\star}\wedge  T < \sigma^{\star,\delta}\wedge \tau \}} e^{-\rho (\tau_M^{\delta}\wedge \tau_M \wedge \sigma^{\star}\wedge  T)} V_r^{\delta}\big(X_{\tau_M^{\delta}\wedge \tau_M \wedge \sigma^{\star}\wedge  T}^{\delta}, R^{\delta}_{\tau_M^{\delta}\wedge \tau_M \wedge \sigma^{\star}\wedge  T}\big)\bigg]. \nonumber 
\end{align}

Recalling \eqref{eq:dynamicsdelta}, thanks to the estimates (i)-(iii) above, the uniform convergence of $V_r^{\delta_n}$ to $V_r$ (cf.\ \eqref{sebseq3}), and the fact that there exists $C_T>0$ such that $\E[\sup_{0 \leq s \leq T}|(X^{\delta_n}_t, R^{\delta_n}_t) - (X_t,r)|^q] \leq C_T\delta_n^{q}$, with $X_t:=X^{0;0}_t$ and $1 \leq q < \infty$, it can be shown that (see Theorem 3.7 in Section 3 of Chapter 3 of Chapter \cite{BL} -- in particular p.\ 322 -- and especially Lemma 4.17 in \cite{ChDeA} for a detailed proof in a related but different setting) $\tau_M^{\delta_n}\wedge \tau_M \wedge \sigma^{\star,\delta_n}\wedge \sigma^{\star}\wedge \tau \wedge T \to \tau_M \wedge \sigma^{\star} \wedge \tau \wedge T$ as $n \uparrow \infty$, $\P$-a.s.
Therefore, taking limits in \eqref{limitunif1} with $\delta=\delta_n$ as $n\uparrow \infty$, using the latter convergence of stopping times and \eqref{sebseq1}-\eqref{sebseq2}, one finds
\begin{align*}
    V_r(x,r)
    &\geq \E\bigg[ \int_0^{\sigma^{\star}\wedge \tau_M \wedge \tau \wedge T} e^{-\rho s} \big(-\theta b V_x(X_s,r) - f_r(X_s,r)\big)\, \d s +e^{-\rho \sigma^{\star} }K \mathds{1}_{\{ \sigma^{\star}< \tau_M \wedge \tau \wedge T\}}\\
  & -e^{-\rho \tau }K\mathds{1}_{\{\tau < \sigma^{\star} \wedge \tau_M \wedge T\}} +e^{-\rho (\tau_M \wedge T) }V_r(X_{\tau_M \wedge T},r)\mathds{1}_{\{\tau_M \wedge \sigma^{\star} \wedge T < \sigma^{\star} \wedge \tau\}}\bigg].
\end{align*}
Letting now $M\uparrow \infty$ and $T\uparrow \infty$ and invoking the dominated convergence theorem we obtain
\begin{align}
\label{Vr1}
    V_r(x,r)
    &\geq \E\bigg[ \int_0^{\sigma^{\star} \wedge \tau} e^{-\rho s} \big(-\theta b V_x(X_s,r) - f_r(X_s,r)\big)~\d s
    +e^{-\rho \sigma^{\star} }K \mathds{1}_{\{ \sigma^{\star}< \tau\}} -e^{-\rho \tau }K\mathds{1}_{\{\tau < \sigma^{\star} \}}\bigg],
\end{align}
for any $\mathbb{F}$-stopping time $\tau$.

Analogously, picking $\zeta=\tau^{\star,\delta_n}\wedge \tau^{\star}\wedge\sigma$, for any $\F$-stopping time $\sigma$, in \eqref{Vrdelta}, and taking limits as $n \uparrow \infty$, and then as $M\uparrow \infty$ and $T\uparrow \infty$, yield
\begin{align}
\label{Vr2}
    V_r(x,r) &\leq \E\bigg[ \int_0^{\sigma \wedge \tau^{\star}} e^{-\rho s} \big(-\theta b V_x(X_s,r) - f_r(X_s,r)\big)\, \d s
    +e^{-\rho \sigma }K \mathds{1}_{\{ \sigma< \tau^{\star}\}} -e^{-\rho \tau^{\star} }K\mathds{1}_{\{\tau^{\star} < \sigma \}}\bigg].
\end{align}

Finally, the choice $\zeta=\tau^{\star,\delta_n}\wedge \tau^{\star} \wedge \sigma^{\star,\delta_n}\wedge \sigma^{\star}$ leads (after taking limits) to
\begin{align}
\label{Vr3}
    V_r(x,r)
    &= \E\bigg[ \int_0^{\sigma^{\star} \wedge \tau^{\star}} e^{-\rho s} \big(-\theta b V_x(X_s,r) - f_r(X_s,r)\big)\, \d s
    +e^{-\rho \sigma^{\star} }K \mathds{1}_{\{ \sigma^{\star}< \tau^{\star}\}} -e^{-\rho \tau^{\star}}K\mathds{1}_{\{\tau^{\star} < \sigma^{\star} \}}\bigg].
\end{align}

Combining \eqref{Vr1}, \eqref{Vr2}, and \eqref{Vr3} completes the proof.



\medskip

\indent \textbf{Acknowledgments.} Financial support by the German Research Foundation (DFG) through the Collaborative Research Centre 1283 is gratefully acknowledged by the authors. The authors also thank Peter Bank, Dirk Becherer, Cristina Caroli Costantini, Peter Frentrup, and Mihail Zervos for interesting discussions.

We are also indebted to three anonymous referees for their pertinent and useful comments and suggestions.


\end{document}